\documentclass[12pt]{amsart}
\usepackage{mathrsfs}
\usepackage{geometry}
\usepackage[all]{xy}

\usepackage{amsmath, amsthm}
\usepackage{amssymb}
\usepackage{amscd}
\usepackage{latexsym}
\usepackage{epsfig}
\usepackage{graphics}
\usepackage{amsfonts}
\usepackage{psfrag}
\usepackage{graphicx}
\usepackage{subfigure}
\usepackage{fancyhdr}

\def\N {{\mathbb{N}}}

\def\R {{\mathbb{R}}}

\def \T{{\mathbb{T}}}
\def\p {{\bf{p}}}
\def\bb{{\bf{b}}}
\def\vv{{\bf{v}}}
\def\S{{\mathcal{S}}}
\def\Ga{{G}}
\def\la{{\lambda}}
\def\ra{{\rightarrow}}
\def\ba{{\backslash}}
\def\cdots {{\cdot\cdot\cdot}}
\newtheorem{theorem}{Theorem}[section]
\newtheorem{question}[theorem]{Question}
\newtheorem*{thm1}{Theorem 1}
\newtheorem*{thm2}{Theorem 2}
\newtheorem{lemma}[theorem]{Lemma}
\newtheorem{proposition}[theorem]{Proposition}
\newtheorem{corollary}[theorem]{Corollary}
\newtheorem{definition}[theorem]{Definition}

\theoremstyle{definition}
\newtheorem{example}[theorem]{Example}
\newtheorem{remark}[theorem]{Remark}
\newtheorem{remark on notation}[theorem]{Remark on Notation}
\newtheorem{notation}[theorem]{Notation}

\numberwithin{equation}{section}

\fancypagestyle{plain}{
\fancyhead{}}

\begin{document}
\title{lower bound for the rank of 2-dimensional generic rigidity matroid for regular graphs of degree four and five}

\author{Shisen Luo}

\address{Department of Mathematics, Cornell University,
Ithaca, NY 14853-4201, USA}

\email{{\tt ssluo@math.cornell.edu}}


\begin{abstract}
In this note we prove a lower bound for the rank of 2-dimensional generic rigidity matroid for regular graphs of degree four and five. Also, we give examples to show the order of the bound we give is sharp.
\end{abstract}

\date{\today}
\maketitle \tableofcontents

\section{\bf Introduction}\label{sec:introduction}
Let $G=(V,E)$ be a connected graph and $\p: V\rightarrow \R^2$ be a generic planar realization. The graph $G$ is always assumed to be finite and simple. We refer to \cite{GSS:Combinatorial Rigidity} for some basic definitions in rigidity theory. Denote by $R(\p)$ the rigidity matrix of G(\p). It is a matrix of size $|E|\times 2|V|$, where $|E|$ and $|V|$ means the number of edges and the number of vertices respectively. The rank of the 2-dimesional generic rigidity matroid of $(V,E)$ can be defined as the rank of the $R(\p)$. We will denote this number by $r(G)$. It is in fact independent of the choice of the generic realization $\p$. When $|V|\geq 3$, it is well known that $r(G)\leq 2|V|-3$. In this note, we study the lower bound for $r(G)$ in the case when $G$ is a regular graph of degree 4 or 5.  

The problem may be interesting in its own right, but a few words are due to explain some hidden interesting aspects of it in that it naturally arises from the field of symplectic geometry. We will only give very brief explanation here, for readers who are interested in the geometric background, some more detailed information are available in \cite{Luo:Betti}.  

Assume $M$ is a compact manifold of dimension $2d$ and there is a two-dimensional torus $\T$ acting on it. People are interested in the {\it equivariant cohomology} $H_{\T}^{*}(M)$ for various reasons. In the paper \cite{GKM}, Goresky, Kottwitz and MacPherson showed that for a certain class of manifolds, which we will just refer to as GKM manifolds, the equivariant cohomology can be computed combinatorially. More concretely, a regular graph $G=(V, E)$ of degree $d$ and a map $\alpha: E\rightarrow \R[x,y]_1$, where $\R[x,y]_1$ means the set of non-zero linear   
polynomials in $x$ and $y$, can be assigned to the manifold $M$. And $H_{\T}^{*}(M)$ is isomorphic to $H_{\T}^{*}(G, \alpha)$ given as below
 \begin{equation}\label{eq:GKM}
 H_{\T}^{*}(G, \alpha)=\left\{(f_1, f_2, ..., f_{|V|})\in \bigoplus_{i=1}^{|V|}\R[x,y]\bigg| \alpha(e_{ij})\big| f_i-f_j, \forall e_{ij}\in E\right\}.
 \end{equation}
 It is called the {\it graph cohomology} of the pair $(G, \alpha)$.
 Their work inspired a lot of subsequent research studying these manifolds, as well as the combinatorial object $H_{\T}^{*}(G, \alpha)$, to name a few, 
 \cite{GZ:graph}, \cite{GZ:one skeleton}, \cite{GKM sheaf} and recently \cite{Morton}, \cite{Luo:Betti}.  The case of particular interest to us is the case we referred to as {\it Hamiltonian GKM manifolds} in \cite{Luo:Betti}. In this case, there exists a map $\p: V\rightarrow \R^2$ such that $\alpha$ is induced from $\p$ in the following sense: if $\p(v_i)=(x_i, y_i), \p(v_j)=(x_j, y_j)$ and $e_{ij}\in E$, then $\alpha(e_{ij})=(x_j-x_i)x+(y_j-y_i)y$. In this case, we write $H_{\T}^{*}(G, \p)$ or $H_{\T}^{*}(G)$ for $H_{\T}^{*}(G, \alpha)$. There are geometric reasons to desire an upper bound for the dimension of $H_{\T}^{1}(G, \p)=\{(f_1, f_2, ..., f_{|V|})\in H_{\T}^{*}(G, \p)\big| f_i \mbox{\ is \ linear\ polynomial\ for\ all\ }i\}$.

{\bf Claim:} If $\p$ is injective, then $\dim H_{\T}^{1}(G, \p)=2|V|-rank(R(\p))$.
\begin{proof}[Proof of the claim:]
$H_{\T}^{1}(G, \p)$ can be viewed as a vector subspace of \[\left\{(a_1y-b_1x, a_2y-b_2x, ..., a_{|V|}y-b_{|V|}x)\big| a_i, b_i\in \R)\right\}\cong \R^{2|V|}.\]
If we let $\p(v_i)=(x_i, y_i)$, the condition $\alpha(e_{ij})\big| f_{i}-f_{j}$ can be translated to 
\[(x_j-x_i)x+(y_j-y_i)y\Big|-(b_i-b_j)x+ (a_i-a_j)y,\]
which in turn is equivalent to 
\[(a_i-a_j)(x_i-x_j)+(b_i-b_j)(y_i-y_j)=0.\]
Rewrite it as 
\[(a_1, b_1, a_2, b_2, ..., a_{|V|}, b_{|V|})\cdot (... ,x_i-x_j, y_i-y_j, ..., x_j-x_i, y_j-y_i, ...),\]
where the only four possibly non-zero entries in the vector on the right are the $2i-1, 2i, 2j-1, 2j$-th entries. We observe that $(..., x_i-x_j, y_i-y_j, ..., x_j-x_i, y_j-y_i, ...)$ is exactly the row vector in the rigidity matrix $R(\p)$ corresponding the edge $e_{ij}$.  So $\dim H_{\T}^{1}(G, \p)=2|V|-rank(R(\p))$.
\end{proof}

So the lower bound for $rank(R(\p))$ will give an upper bound for $\dim H_{\T}^{1}(G, \p)$, which in turn will have interesting geometric consequences. 

\begin{remark}
In order to study $H_{\T}^{k}(G, \p)$ for $k\geq 2$, a generalized version of rigidity matrix was defined in \cite{Luo:Betti}. The view of graph cohomology might provide a different and interesting perspective on facts in rigidity theory as well.  
\end{remark}

\begin{remark}
In fact, most frameworks $G(\p)$ arising from a Hamiltonian GKM manifolds are not generic, what we can say about them is that they are locally in general position, i.e., the edges incident to the same vertex are in pairwise linearly independent directions. For more combinatorial constraints about the frameworks arising from GKM manifolds, one can consult \cite{GZ:graph}, \cite{GZ:one skeleton} or \cite{Luo:Betti}. In \cite{Luo:Betti}, it was shown that for a framework arising from a Hamiltonian GKM manifold, if it is in general position (see Definition~\ref{def:general}), then the graph must be $d$-edge-connected, where $d$ is the degree of the graph.  This motivates us to study the lower bound of $rank(R(\p))$ in the case when $\p$ is in general position and $G$ is a $d$-edge-connected graph. The $d=4$ case will be addressed in a subsequent paper \cite{Luo:Rigidity}.  
\end{remark}

Now let's state the main results in this paper.
\begin{thm1}
Let $G=(V,E)$ be a connected regular graph of degree 4, then we have 
\[r(G)\geq \dfrac{8}{5}|V|-1.\]
\end{thm1}

\begin{figure}[ht]
\centering
\subfigure[]{
\includegraphics[scale=0.6]{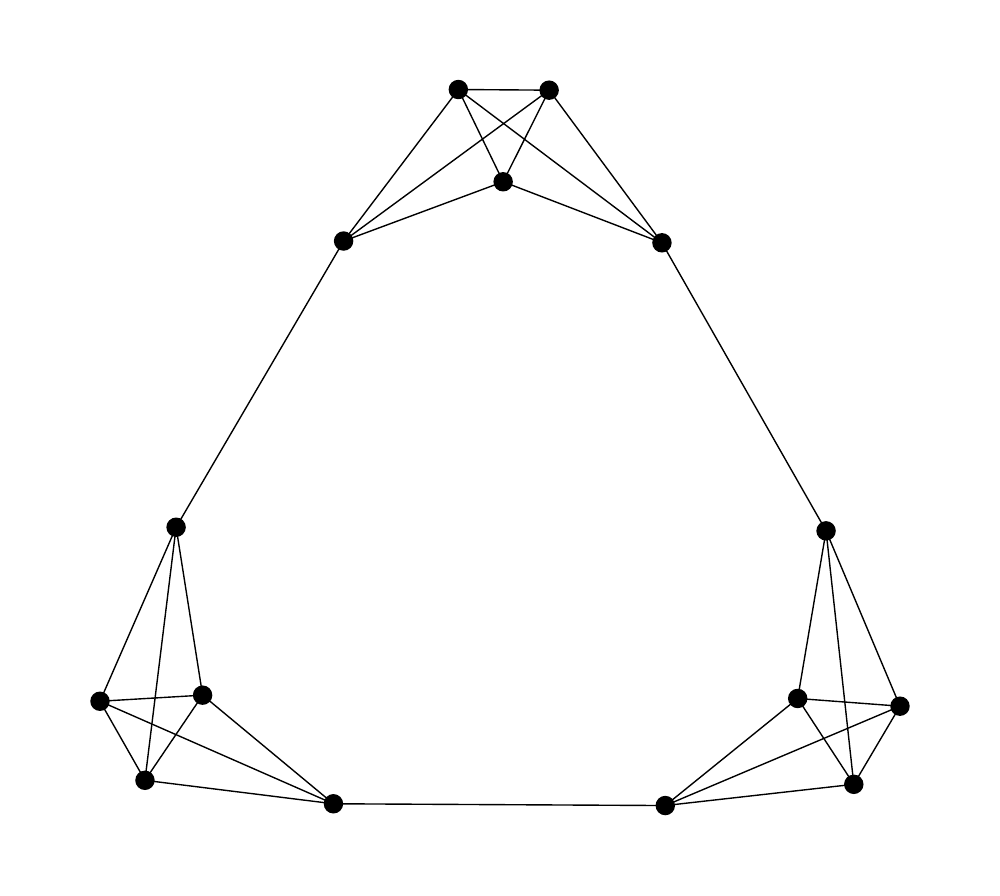}
\label{fig:15Vertices}}
\subfigure[]{
\includegraphics[scale=0.6]{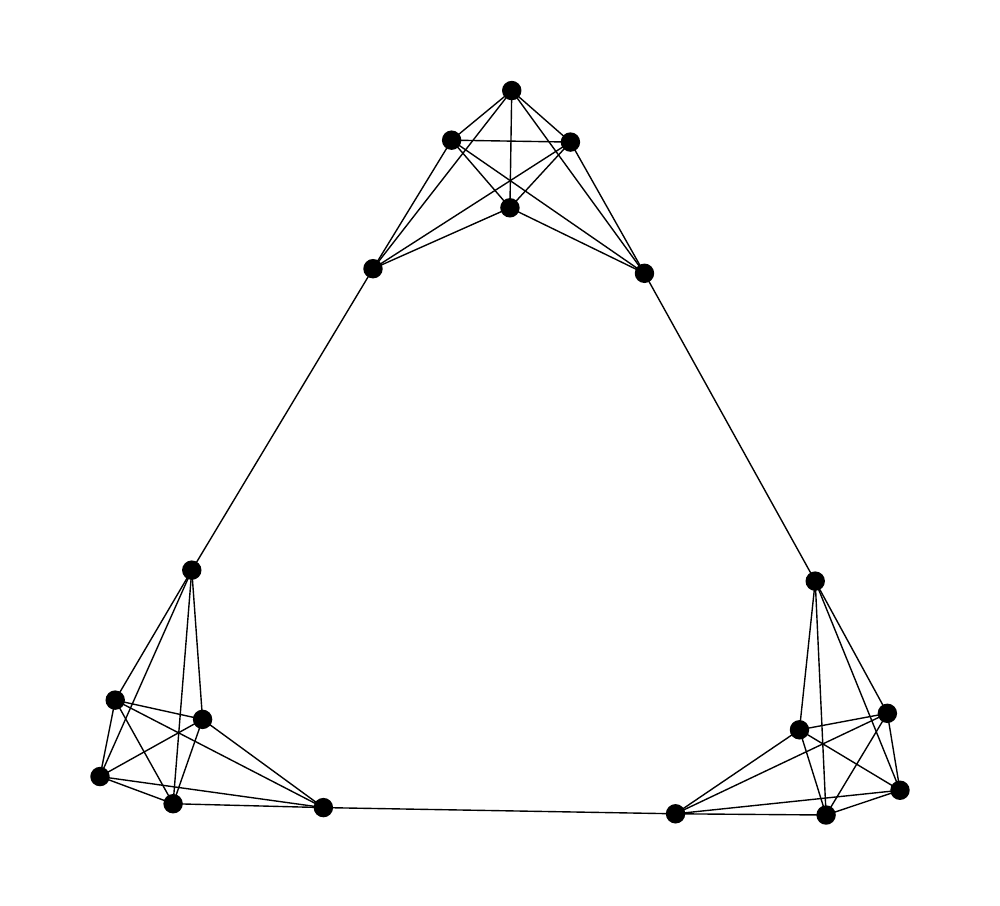}
\label{fig:18Vertices}}
\caption{Examples of regular graphs}
\end{figure}

\begin{example}\label{ex:15vertices}
We construct a regular graph of degree four with $15$ vertices as follows. Take three copies of complete graphs on $5$ vertices. Delete one edge from each, then connect the remaining graphs to form a single 4-valent regular graph.  One planar realization of the graph is illustrated in Figure~\ref{fig:15Vertices}. 
It can be shown that the rank of the generic rigidity matroid for this graph is 24. This example can be easily generalized to a regular 4-valent graph with $5k$ vertices for any $k\geq 2$, whose generic rigidity matroid has rank $8k$.  So for this classes of graphs, $r(G)= \frac{8}{5}|V|$. So the order of the lower bound we gave in Theorem 1 is  sharp. Also it is easy to see the equality in Theorem 1 holds in the case of complete graph on $5$ vertices. 
\end{example}

%\begin{figure}[!h]
%\includegraphics[height=60mm]{15Vertices.pdf}
%\end{figure}

\begin{thm2}
Let $G=(V,E)$ be a connected regular graph of degree 5, then we have 
\[r(G)\geq \dfrac{5}{3}|V|-1.\]
\end{thm2}

Similar to Example~\ref{ex:15vertices} we can construct a class of 5-valent graphs which demonstrates the order of the above bound is sharp. 
\begin{example}\label{ex:18vertices}
Take $k$ copies of complete graphs on $6$ vertices, $k\geq 2$. Delete one edge from each, then connect the remaining graphs to form a single 5-valent graph. For $k=3$, one planar realization of the graph is illustrated in Figure~\ref{fig:18Vertices}. It can be shown that the rank of the generic rigidity matroid for this graph is $10k$. So for this classes of graphs, $r(G)= \frac{5}{3}|V|$, hence shows the order of the lower bound we gave in Theorem 2 is sharp. Also it is easy to see the equality in Theorem 2 holds in the case of complete graph on $6$ vertices.
\end{example}

\begin{remark}
When $G=(V,E)$ is regular of degree 3, then it can be easily shown that when $|V|\geq 4$, we have $r(G)=|E|$. And when $|V|=4$, we have $r(G)=5$.  
\end{remark}

\begin{question}
Does similar result hold for regular graphs of degree $d$, $d\geq 6$?
\end{question}

One notion related to generic realization is {\it configuration in general position}.

\begin{definition}\label{def:general}
Given $G=(V,E)$, we call a map $\p : V\rightarrow \R^2$ a planar configuration in general position if no three points in $\p(V)$ lie on the same line. In particular, $\p$ is injective. The rank of the infinitesimal rigidity matroid of $\p$ is defined to be $rank(R(\p))$. We denote this number by $r(G(\p))$.
\end{definition}

\begin{remark}
As the notation has already suggested, $r(G(\p))$ does not only depend on $G$, but also on $\p$. For any $\p$ a planar configuration in general position, we have $r(G(\p))\leq r(G)$.
\end{remark}

\begin{question}
If $G=(V,E)$ is regular graph of degree four and $\p: V\rightarrow \R^2$ is a planar configuration in general position, what is the lower bound for $r(G(\p))$? Does the bound given in Theorem 1 still hold? 
\end{question}
An affirmative answer to this question will be given in a subsequent paper \cite{Luo:Rigidity}.

\begin{question}
In the examples we provided, the graph can become disconnected upon deleting two edges. If we impose "higher connectivity" upon the graph, say, the graph remains connected upon deleting any three edges, can the bound be improved?
\end{question}
Again, in the case of $4$-valent graphs, this question will be addressed in \cite{Luo:Rigidity}.

{\bf Acknowledgement:} I would like to take this chance to thank Robert Connelly, Tara Holm and Edward Swartz for many helpful discussions.
\section{\bf Preliminaries and Preparations}
 Let $G(\p)$, where $G=(V,E)$, be a graph with generic planar realization. We will use $m$ to stand for $|V|$, the number of vertices. The vertices are numbered as $v_{1},v_2,\cdots, v_m$. The edge connecting $v_{i}$ and $v_{j}$ will be denoted by $e_{ij}$. We do not distinguish between $e_{ij}$ and $e_{ji}$ but normally make the first coordinate smaller than the second one. We can view an edge $e_{ij}$ as a unordered pair of vertices $(v_i, v_j)$ and we sometimes informally say $e_{ij}$ contains $v_{i}$. 
 %Denote $\p(v_{i})$ by $\p_{i}$ and 
 Assume $\p(v_i)=(x_{i}, y_{i})$.  We will talk about linear algebra in $\R^{2m}$ a lot and it would be handy sometimes to use standard basis to express vectors. We use $\bb_{i}$ to stand for the $i$-th standard basis. To each edge $e_{ij}$, we can associate it with an vector in $\R^{2m}$, given by \[(x_{i}-x_{j})\bb_{2i-1}+(y_{i}-y_{j})\bb_{2i}+(x_{j}-x_{i})\bb_{2j-1}+(y_{j}-y_{i})\bb_{2j}.\] We denote this vector by $\vv_{ij}$. If $F\subset E$, we will use $<F>$ to denote the subspace of $\R^{2m}$ spanned by $\{{\bf v}_{ij}\big| e_{ij}\in F\}$.
 
 The rigidity matrix $R_{G}(\p)$ (we deliberately added $\Ga$ as subscript as we will talk about graphs with same map $\p: V\rightarrow \R^2$ but with different edge sets) is an $|E|\times 2m$ matrix whose rows are indexed by the edge set, and the row corresponding the edge $e_{ij}$ is given by $\vv_{ij}$.  The subscript $G$ is sometimes omitted when there is no possible confusion. The rank of this matrix is by definition $r(G)$. For any $F\subseteq E$, we denote by $\mathcal{S}(F)$ the set  of linear relations among $\{\vv_{ij}\big|e_{ij}\in F\}$, i.e.
 \[\S(F)=\{\omega: F\rightarrow \R: \sum_{e_{ij}\in F}\omega(e_{ij})\vv_{ij}=0\}.\] 
This is the collection of {\it resolvable stresses} of $F$.  Let $s_{\bf{p}}(F)=dim\S(F)$, and call it the {\it number of stress} of $F$. When $F=E$, we also use $\S(G)$ to stand for $\S(E)$ and use $s_{\bf{p}}(G)$ for $s_{\bf{p}}(E)$. It follows from simple linear algebra that $r(G)=|E|-s_{\bf{p}}(G)$.  
 Although we are only concerned with generic rigidity, hence generic realization in this paper, at one point we would need to consider a non-generic realization. We point out here that the definition of $s_{\bf{p}}(F)$ carries over to case when $\bf{p}$ is not generic without difficulty.
 One can show $s_{\bf{p}}(F)$ does not depend on the choice of $\p$ as long as it is generic, so in the case of generic realization, we write $s(F)$ for $s_{\bf{p}}(F)$ and $s(G)$ for $s_{\p}(G)$. 

\begin{definition}
Let $G=(V,E)$ be a graph, the degree of a vertex $v_{i}$ is defined as the number of edges containing $v_{i}$, we denote this number by $\lambda(v_{i})$.
\end{definition}

The following lemma about degree will be used in Section~\ref{section:four valent graph} in the proof of Lemma~\ref{lemma:fourValent}.

\begin{lemma}\label{lem:connect}
For any connected graph $G=(V,E)$, we have 
\[\sum_{v_{i}\in V}(\lambda(v_{i})-2)\geq -2.\]
\end{lemma}
\begin{proof}
If the graph is a tree, then we can show the equality holds by an induction on the number of vertices.  In general, a graph always has a spanning tree, so the inequality holds. 
\end{proof}

The following simple lemma and its corollary will be used repeatedly in the following sections, and we would call it  the {\it Deleting Lemma}.

\begin{lemma}[Deleting Lemma]
Given $G=(V,E)$, assume there is a vertex $v_{i}$, such that $\lambda(v_{i}) = 2$. Let $E_{v_i}$ be the set of edges that contains $v_{i}$ and $E'=E\backslash E_{v_i}$. Then $s(E')=s(E)$. 
\end{lemma}

\begin{proof}
Without loss of generality, we may assume $v_{i}=v_{1}$, and $E_{v_i}=\{e_{12},e_{13}\}$. Pick a generic realization $\p: V\rightarrow \R^2$, and assume there is a dependence relation
\[\sum_{e_{kj}\in E} \omega_{kj}\vv_{kj}=0.\]
If we restrict our attention to the first two coordinates, we see that $\omega_{12}=\omega_{13}=0$. So 
\[\sum_{e_{kj}\in E'} \omega_{kj}\vv_{kj}=0.\]

This says any dependence relation among $\{\vv_{kj}\big| e_{kj}\in E \}$ is in fact a dependence relation  among $\{\vv_{kj}\big| e_{kj}\in E'\}$, so $s(E')=s(E)$.
The proof of a more general statement can be found in Lemma 2.5.6 in \cite{GSS:Combinatorial Rigidity}. 
\end{proof}

\begin{corollary}\label{cor:delete}
Given $G=(V, E)$ and $v_{i}\in V$. Let $E_{v_i}\subseteq E$ be the set of edges that contains $v_{i}$ and $E'=E\backslash E_{v_i}$. If $|E_{v_i}|\leq 2$, then $s(E')=s(E)$. If $|E_{v_i}|\geq 3$, then $s(E)-(|E_{v_i}|-2)\leq s(E')\leq s(E)$. 
\end{corollary}
\begin{proof}
This is straightforward from the Deleting Lemma.
\end{proof}

\begin{notation}
Given $G(\p)$ with $G=(V,E)$, we say a vector ${\bf a}=(a_{1}, a_{2}, ... , a_{2m})\in \R^{2m}$ vanishes on $v_{i}$, or ${\bf a}\big|_{v_{i}}=0$, if $a_{2i-1}=a_{2i}=0$.  We say $\bf{a}$ vanishes on a set $U\subseteq V$ if $\bf{a}$ vanishes on every point in $U$. We denote by $W_{U}$ the set of vectors that vanishes on the $V\backslash U$.  There is natural projection map $P_{U}: \R^{2m}\rightarrow W_{U}$ which sets the coordinates corresponding to $V\backslash U$ to 0. We easily see that $\vv_{ij}\in W_{\{v_i, v_j\}}$.

Given subset $U\subseteq V$, we use $K(U)$ to denote the edge set of the complete graph on $U$. Note that if $U$ consists of one vertex, then $K(U)=\emptyset$.

%Given subset $F\subseteq K(V)$, we use $supp(F)$ to denote the set of vertices that is contained in some edge in $F$, and call it the support of $F$. 
\end{notation}

\begin{proposition}\label{prop: restriction}
Given $G(\p)$ with $G=(V,E)$ and $U\subseteq V$ a nonempty subset, we have 
\[< E >\cap\ W_{U}\ \subseteq\ < K(U) >.\]
\end{proposition}
\begin{proof}
It is enough if we can show 
\[<K(V)>\cap\ W_{U} = <K(U)>.\]
It is clear that in the above formula, the RHS is a subset of the LHS, so it suffices to show 
\begin{equation}\label{eq:1}
dim(< K(V) >\cap\ W_{U})= dim(<K(U)>).
\end{equation}
When $|U|=1$, $<K(V)>\cap\ W_{U}=0$. To see this, notice every vector $\vv_{ij}=(a_1, a_2, ..., a_{2m})$ that corresponding to the edge $e_{ij}$ has the property that $a_{1}+a_{3}+\cdots +a_{2m-1}=0$ and $a_{2}+a_{4}+\cdots +a_{2m}=0$. Therefore any vector in $<K(V)>$ also has this property.  So the only intersection it could have with $W_{U}$ is $\bf{0}$.  Then \eqref{eq:1} clearly holds. 

When $U=V$, \eqref{eq:1} clearly holds.

Now we assume $|U|\geq 2$ and $U\subsetneq V$.  Without loss of generality,  we may assume $v_{1}, v_{2} \in U$. For any point $v_{s}\in V\backslash U$,  we wish to show 
\begin{equation}
W_{\{v_{s}\}} \subseteq <K(V)> + W_{U}.
\end{equation}

For simplicity and without loss of generality, consider $s=3$. Assume 
$$
\begin{array}{ll}
\vv_{13}&=(x_1-x_3,  y_1-y_3,  0,  0,  x_3-x_1,  y_3-y_1, 0, \cdots , 0)\\
&= (x_{1}-x_3)\bb_1+(y_1-y_3)\bb_2+(x_3-x_1)\bb_5+(y_3-y_1)\bb_6
\end{array}$$
 and $$\vv_{23}=(x_2-x_3)\bb_3+(y_2-y_3)\bb_4+(x_3-x_2)\bb_5+(y_3-y_2)\bb_6.$$ Then $$(x_3-x_1)\bb_5+(y_3-y_1)\bb_6\in <K(V)>+W_U$$ and $$(x_3-x_2)\bb_5+(y_3-y_2)\bb_6\in <K(V)>+W_U.$$
 These two vectors span $W_{\{v_3\}}$. So $W_{\{v_{3}\}} \subseteq <K(V)> + W_{U}$.  It follows that $$<K(V)>+W_{U}=\R^{2m}$$ and hence $dim (<K(V)>+W_{U})=2m$. 
 So $$\begin{array}{ll}
 dim (<K(V)\cap\ W_U>)&=dim(<K(V)>)+dim(W_U)-2m\\
 &=2|V|-3+2|U|-2m\\
 &=2|U|-3\\
 &=dim(<K(U)>)
 \end{array}$$
 We used the fact that $dim (<K(U)>)=2|U|-3$ when $|U|\geq 2$.  So \eqref{eq:1} holds and the proof is complete.
\end{proof}

The following lemma will also be used in the following sections repeatedly and we call it the {\it Disconnecting Lemma}, as it studies the rank of the rigidity matrix when the graph become disconnected upon deleting certain edges. 
\begin{lemma}[Disconnecting Lemma]\label{lem:disc}
Assume $G(\p)$, where $G=(V,E)$, is a connected graph with a generic planar realization. Assume upon removing $k$ edges $e_{i_{1}j_{1}}, e_{i_2j_2}, ... , e_{i_kj_k}$ the graph becomes the disjoint union of two connected graphs $G_1=(V_1, E_1)$ and $G_2=(V_2, E_2)$, i.e., $V_{1}\cap V_2=\emptyset$,  $V_{1}\cup V_2=V$, and  $E_1\cup E_2\cup \{e_{i_1j_1}, ..., e_{i_kj_k}\}=E$.  

We let $V_3=\{v_{i_1}, v_{i_2}, ..., v_{i_k}\}$, $V_4=\{v_{j_1}, v_{j_2}, ..., v_{j_k}\}$(the vertices are allowed to repeat and the repeated vertices should only be counted once in the set) and $E_3=\{e_{i_1j_1}, e_{i_2j_2}, ..., e_{i_kj_k}\}$. Assume $V_3\subset V_1$ and $V_4\subset V_2$. We form a new graph $G_5=(V_5, E_5)$ by letting $V_5=V_{3}\cup V_{4}$ and $E_5=K(V_3)\cup K(V_4)\cup E_3$.

If $s(E_{5})=0$, then $s(G)=s(G_1)+s(G_2).$ 
\end{lemma}
\begin{proof}
Suppose there is a linear relation
\begin{equation}\label{eq:4}
\sum_{e_{ij}\in E}\omega_{ij}\vv_{ij}=0.
\end{equation}
We can break the LHS into three parts to get
\begin{equation}\label{eq:3}
\sum_{e_{ij}\in E_1}\omega_{ij}\vv_{ij}+\sum_{e_{ij}\in E_2}\omega_{ij}\vv_{ij}+\sum_{e_{ij}\in E_3}\omega_{ij}\vv_{ij}=0.
\end{equation}
Apply $P_{V_{1}\backslash V_3}: \R^{2m}\rightarrow W_{V_{1}\backslash V_3}$ to \eqref{eq:3} to get 
\[P_{V_1\ba V_3}(\sum_{e_{ij}\in E_1}\omega_{ij}\vv_{ij})=0.\]
Hence $\displaystyle{\sum_{e_{ij}\in E_1}\omega_{ij}\vv_{ij}\in W_{V_3}}$.  Now we apply Proposition~\ref{prop: restriction} to $G_1$ and $V_{3}$, we see that
\[\sum_{e_{ij}\in E_1}\omega_{ij}\vv_{ij}=\sum_{e_{ij}\in K(V_{3})}u_{ij}\vv_{ij}\]
for some constants $u_{ij}\in \R$.
By a similar argument, we see that
\[\sum_{e_{ij}\in E_2}\omega_{ij}\vv_{ij}=\sum_{e_{ij}\in K(V_4)}z_{ij}\vv_{ij}\]
for some constants $z_{ij}\in \R$. 
Then it follows from \eqref{eq:3} that
\[\sum_{e_{ij}\in K(V_{3})}u_{ij}\vv_{ij}+\sum_{e_{ij}\in K(V_4)}z_{ij}\vv_{ij}+\sum_{e_{ij}\in E_3}\omega_{ij}\vv_{ij}=0.\]
But $s(E_5)=0$, this forces $\omega_{ij}=0$ for all $e_{ij}\in E_3$ and 
\[\sum_{e_{ij}\in E_1}\omega_{ij}\vv_{ij}=0,\]
\[\sum_{e_{ij}\in E_2}\omega_{ij}\vv_{ij}=0.\]
So a linear relation \eqref{eq:4}  among $\{\vv_{ij}\big| e_{ij}\in E\}$ is always the sum of a linear relation among $\{\vv_{ij}\big| e_{ij}\in E_1\}$ and a linear relation among $\{\vv_{ij}\big| e_{ij}\in E_2\}$.  So $s(E)=s(E_1)+s(E_2)$, i.e., $s(G)=s(G_1)+s(G_2)$.
\end{proof}

\begin{corollary}\label{cor:disconnect1}
Given $G(\p)$, assume deleting one edge $e_{ij}$ increases the number of connected component of the graph by $1$.  Then deleting this edge does not affect the number of stress.
\end{corollary}
\begin{proof}
This is a straightforward application of the Disconnecting Lemma.
\end{proof}

\begin{corollary}\label{cor:disconnect2}
Given $G(\p)$, assume deleting any one edge would not increase the number of connected components of the graph, but deleting some two edges $e_{ij}$ and $e_{ks}$ increases the number of connected components of the graph by $1$. Then deleting these two edges does not affect the number of stress.
\end{corollary}
\begin{proof}
Also follows from straightforward application of the Disconnecting Lemma.
\end{proof}

The following example demonstrates how we may apply the Deleting Lemma and the Disconnecting Lemma to determine the number of stress of a graph. 
\begin{example}\label{ex:demon}
\begin{figure}[!ht]\label{fig:abcdef}
\centering
\subfigure[]{
\includegraphics[scale=0.4]{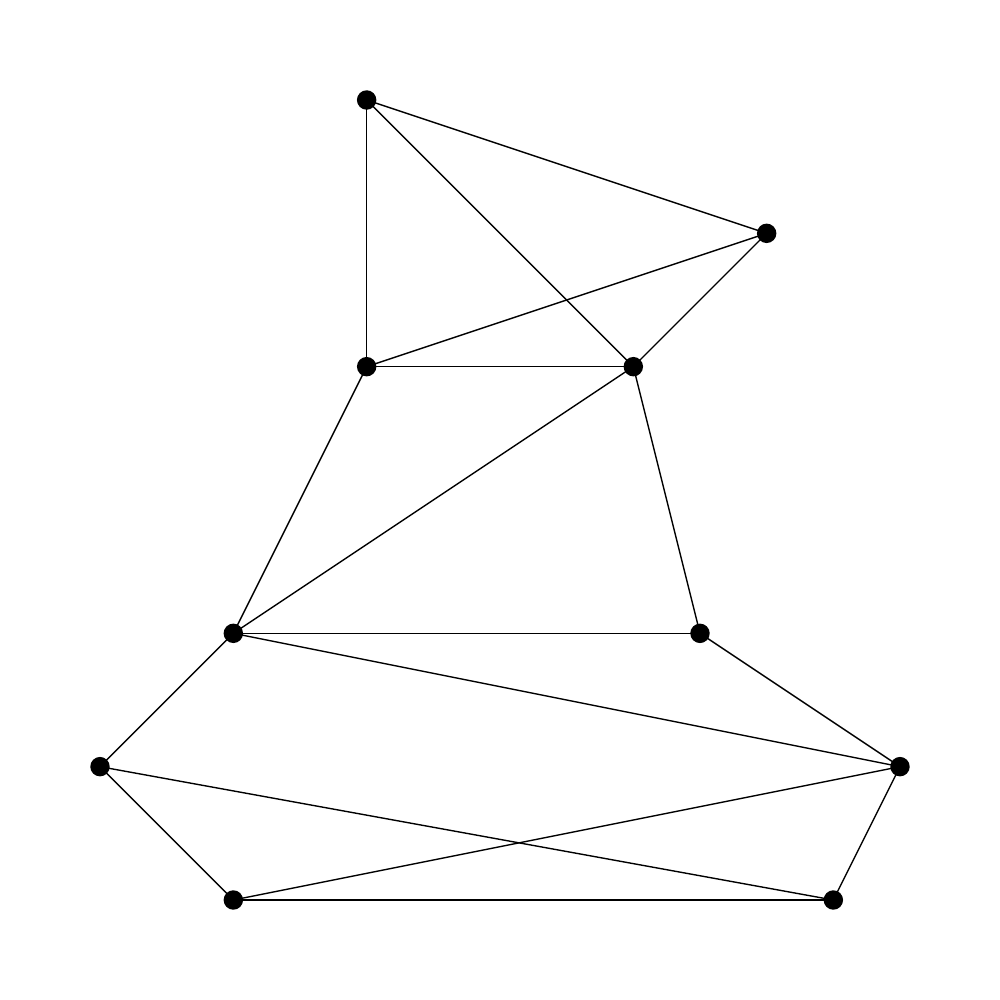}
\label{fig:grapha}}
\subfigure[]{
\includegraphics[scale=0.4]{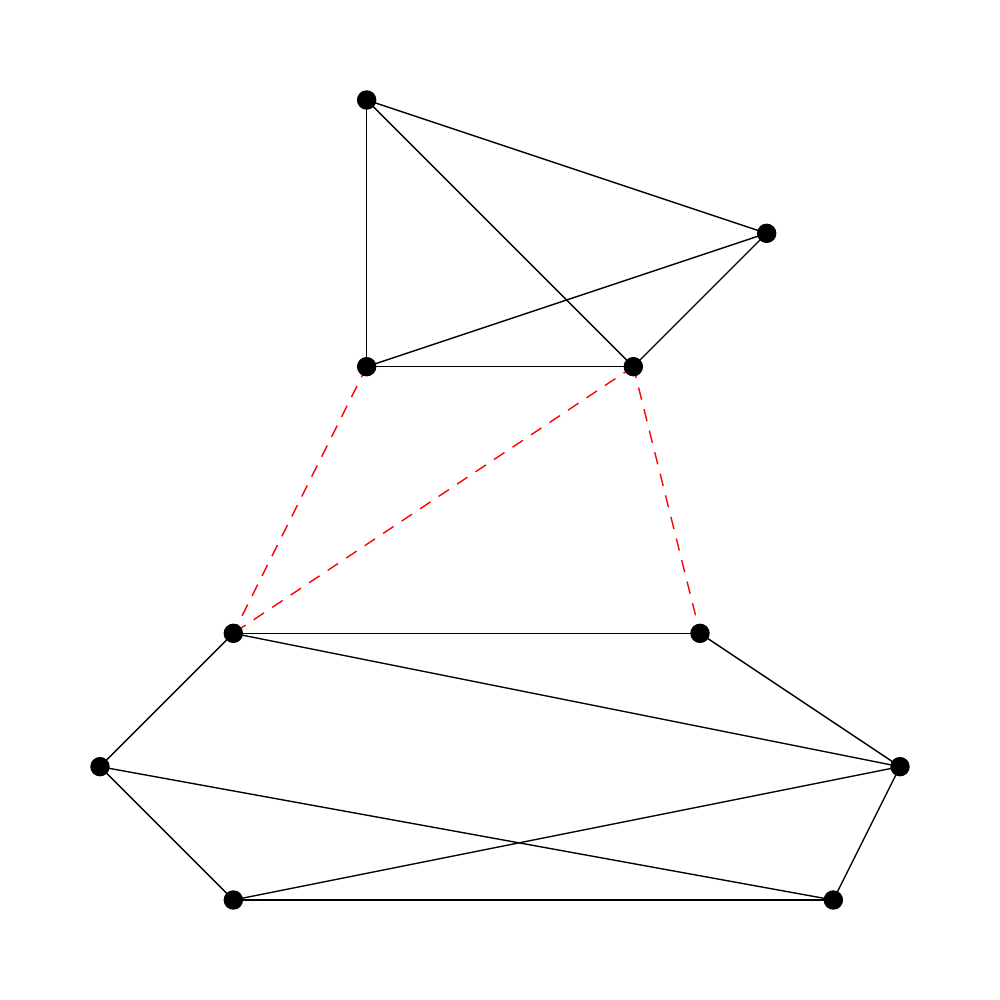}
\label{fig:graphb}}
\subfigure[]{
\includegraphics[scale=0.4]{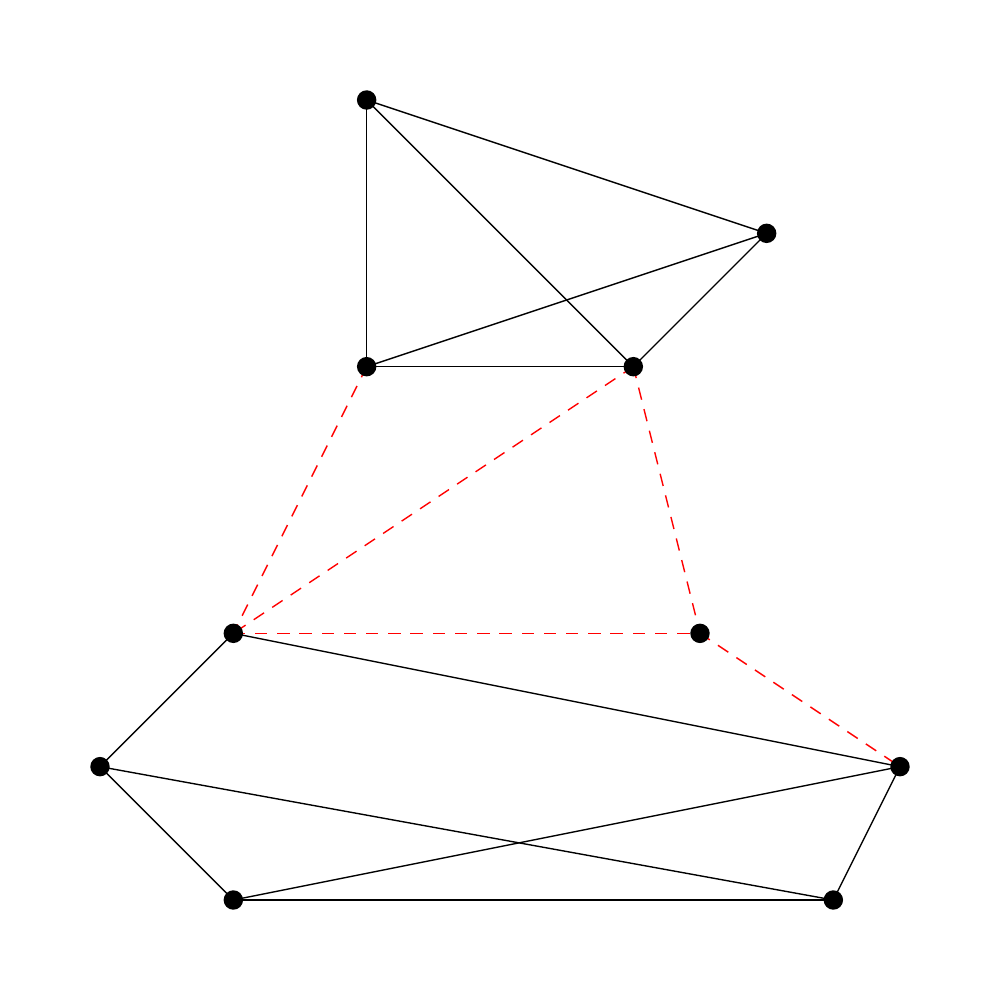}
\label{fig:graphc}}
\subfigure[]{
\includegraphics[scale=0.4]{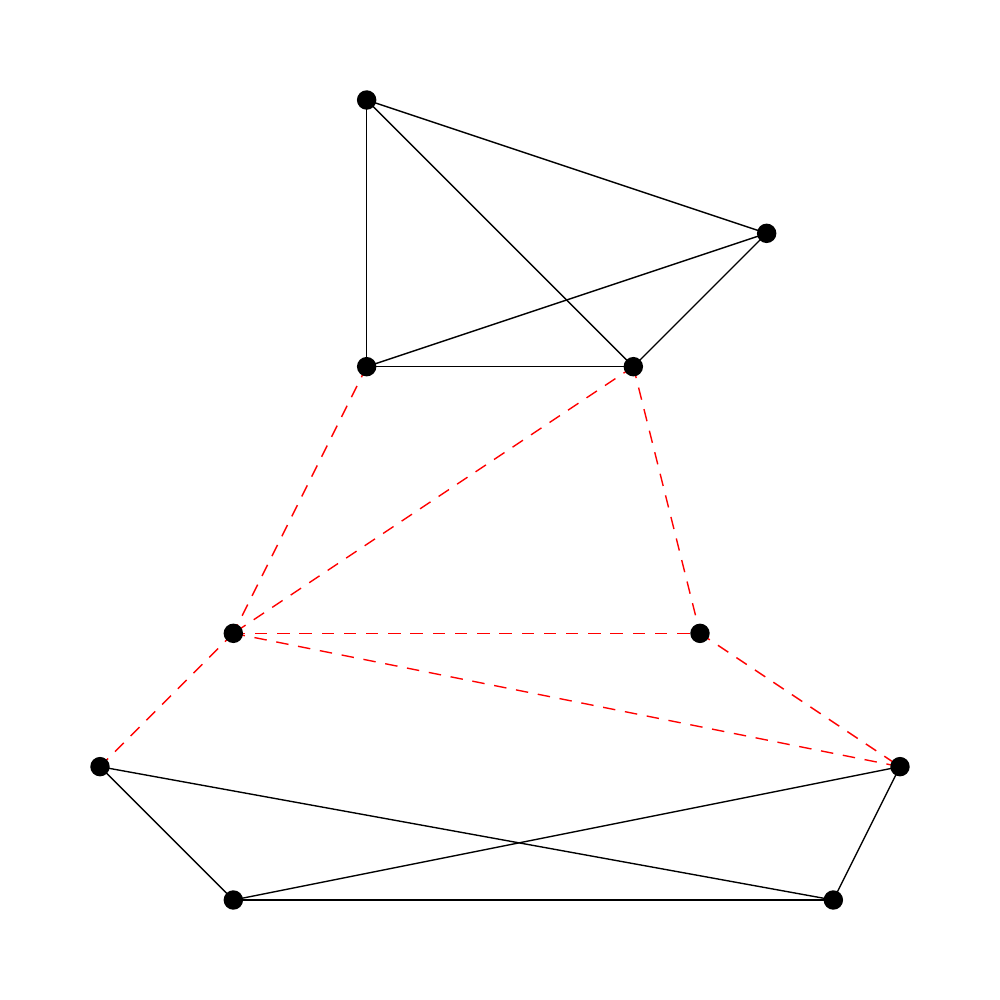}
\label{fig:graphd}}
\subfigure[]{
\includegraphics[scale=0.4]{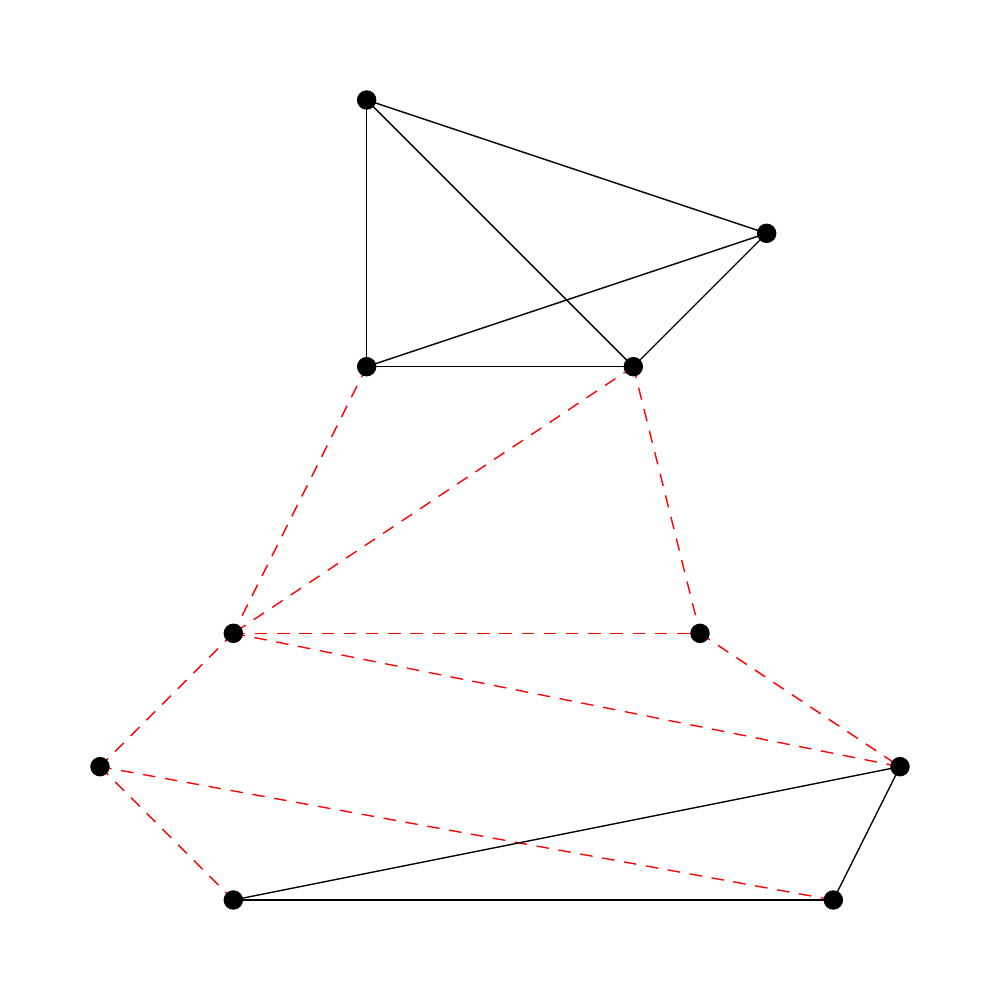}
\label{fig:graphe}}
\subfigure[]{
\includegraphics[scale=0.4]{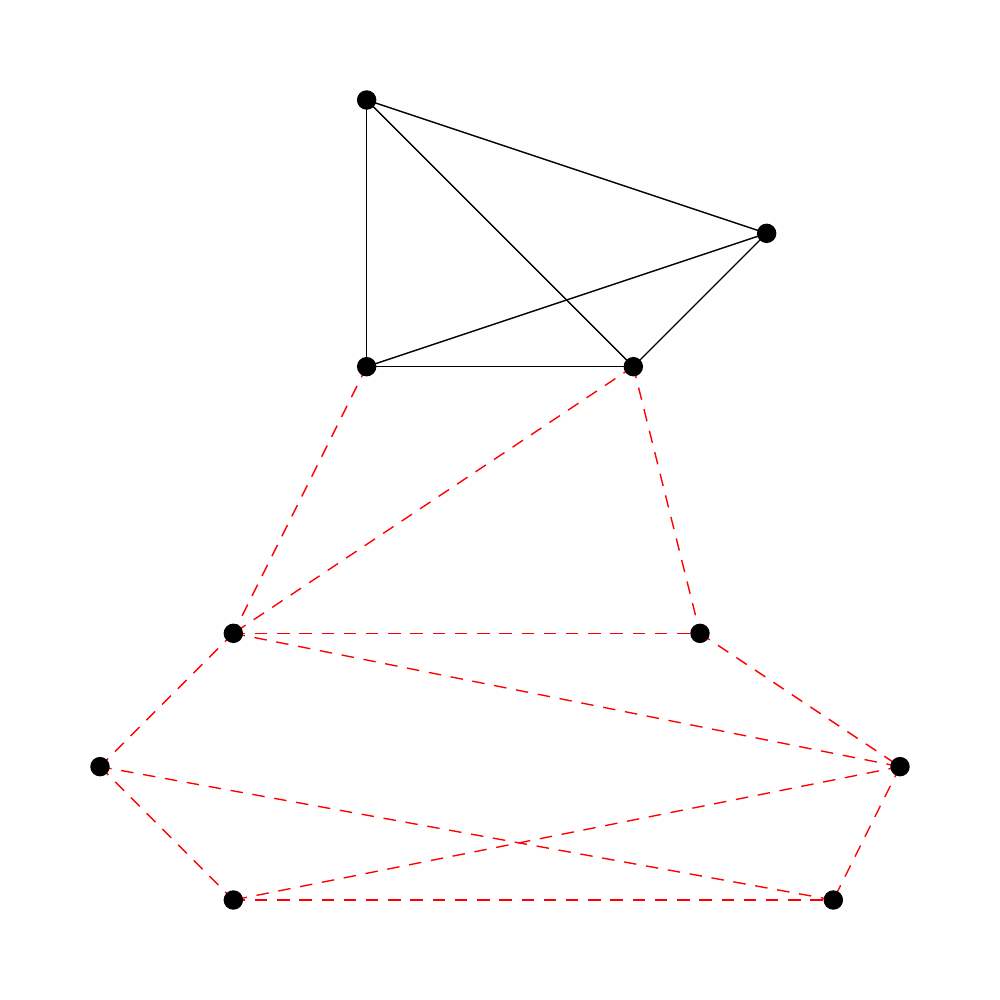}
\label{fig:graphf}}
\caption{A graph with number of stress 1}
\end{figure}
Figure \ref{fig:grapha} is one planar realization of a graph with $10$ vertices.  First we apply the Disconnecting Lemma to delete $3$ edges to obtain graph in Figure \ref{fig:graphb}. We use red dashed lines to denote the edges that have been deleted.  Then we apply the Deleting Lemma repeatedly to the lower graph. Finally we end up with a complete graph on $4$ vertices as in Figure \ref{fig:graphf}, which has number of stress $1$. Each step keeps the number of stress unchanged, so the original graph has number of stress $1$.
\end{example}

\begin{remark}
Up to now, everything we have talked about would still apply if we are talking about "configuration in general position" in place of "generic realization".
\end{remark}

\begin{proposition}\label{prop: one extension}
Given $G=(V, E)$, assume there is a vertex $v_{s}$ with $\lambda(v_s)=3$ and the three edges containing $v_s$ are $e_{si}, e_{sj}, e_{sk}$. Assume $e_{ij}\notin E$. We define a new graph $G'=(V, E')$ by taking $E'=(E\cup \{e_{ij}\})\ba\{e_{si}, e_{sj}, e_{sk}\}$.  Then $s(G)\leq s(G')$.
\end{proposition}
\begin{proof}
\begin{figure}[!ht]\label{fig:one-extension}
\centering
\subfigure[$G(\p)$]{
\includegraphics[scale=0.4]{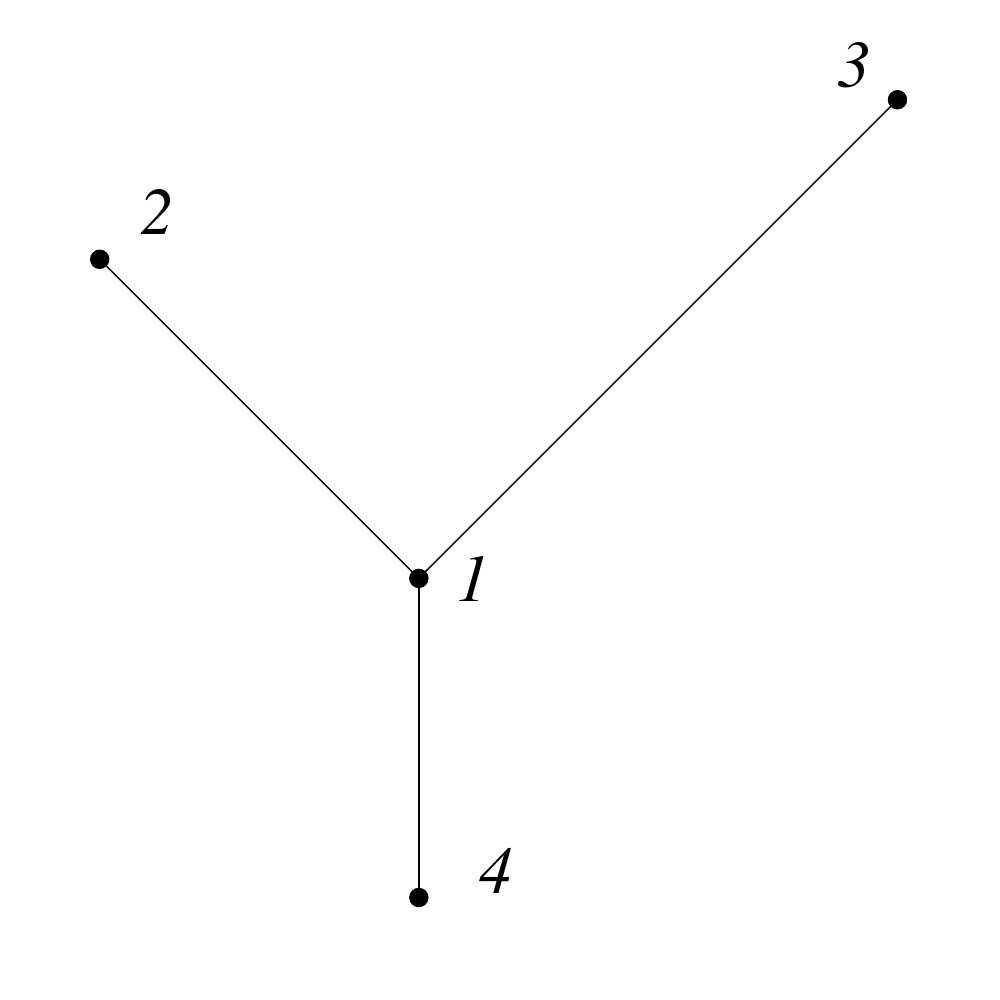}
\label{fig:graph1a}}
\subfigure[$G(\p')$]{
\includegraphics[scale=0.4]{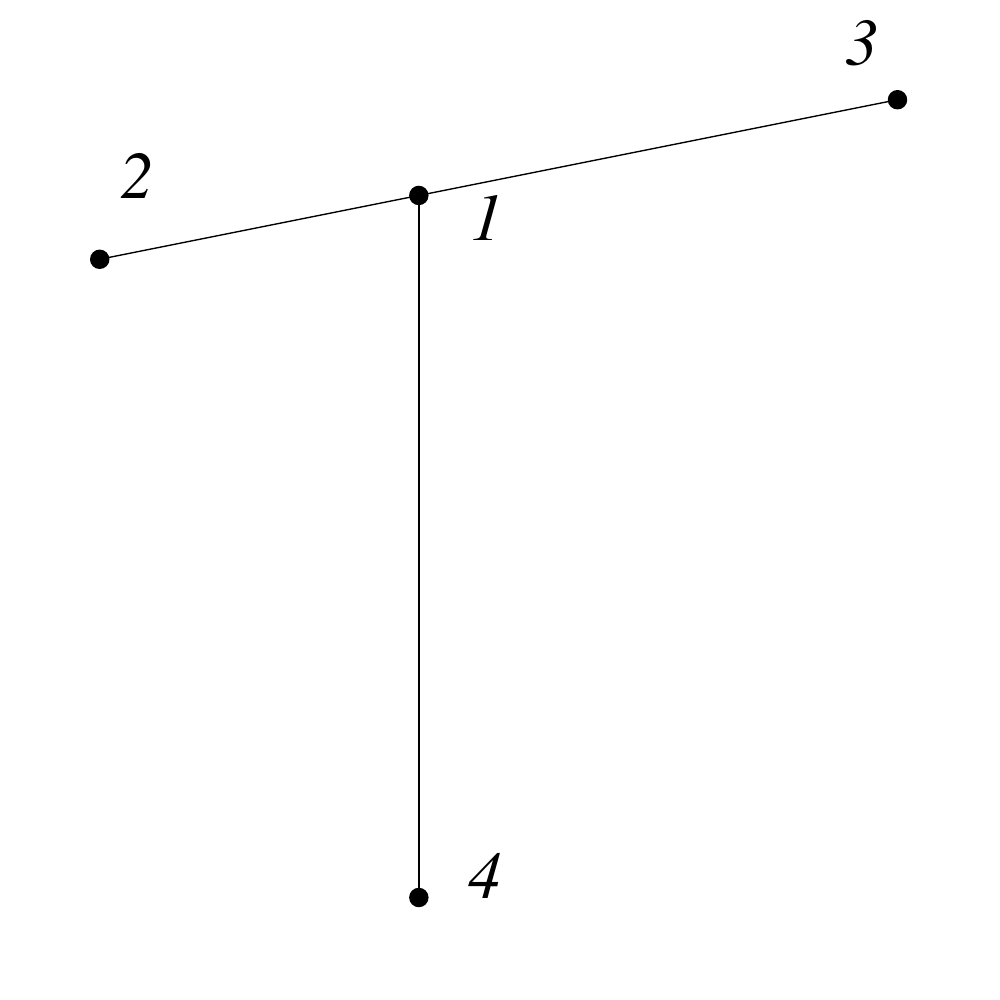}
\label{fig:graph1b}}
\subfigure[$G_1(\p')$]{
\includegraphics[scale=0.4]{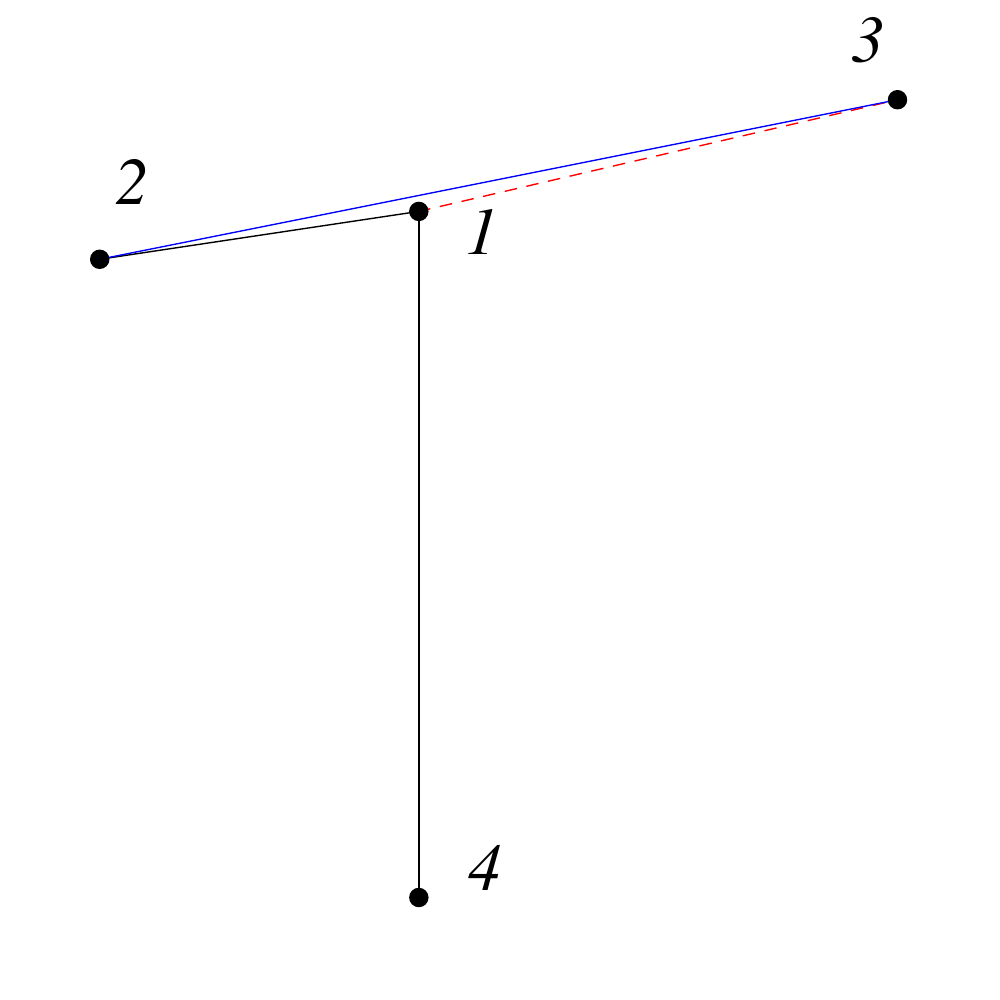}
\label{fig:graph1c}}
\subfigure[$G'(\p')$]{
\includegraphics[scale=0.4]{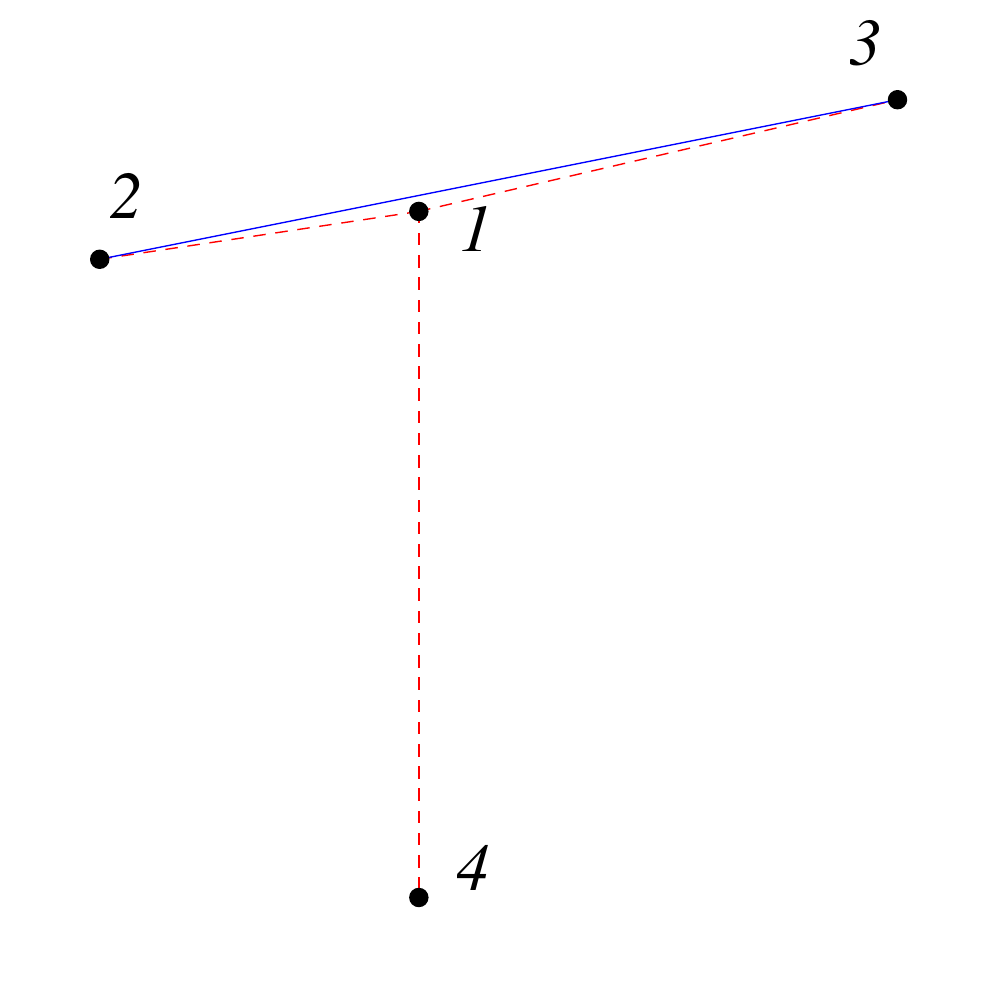}
\label{fig:graph1d}}
\subfigure[$G'(\p)$]{
\includegraphics[scale=0.4]{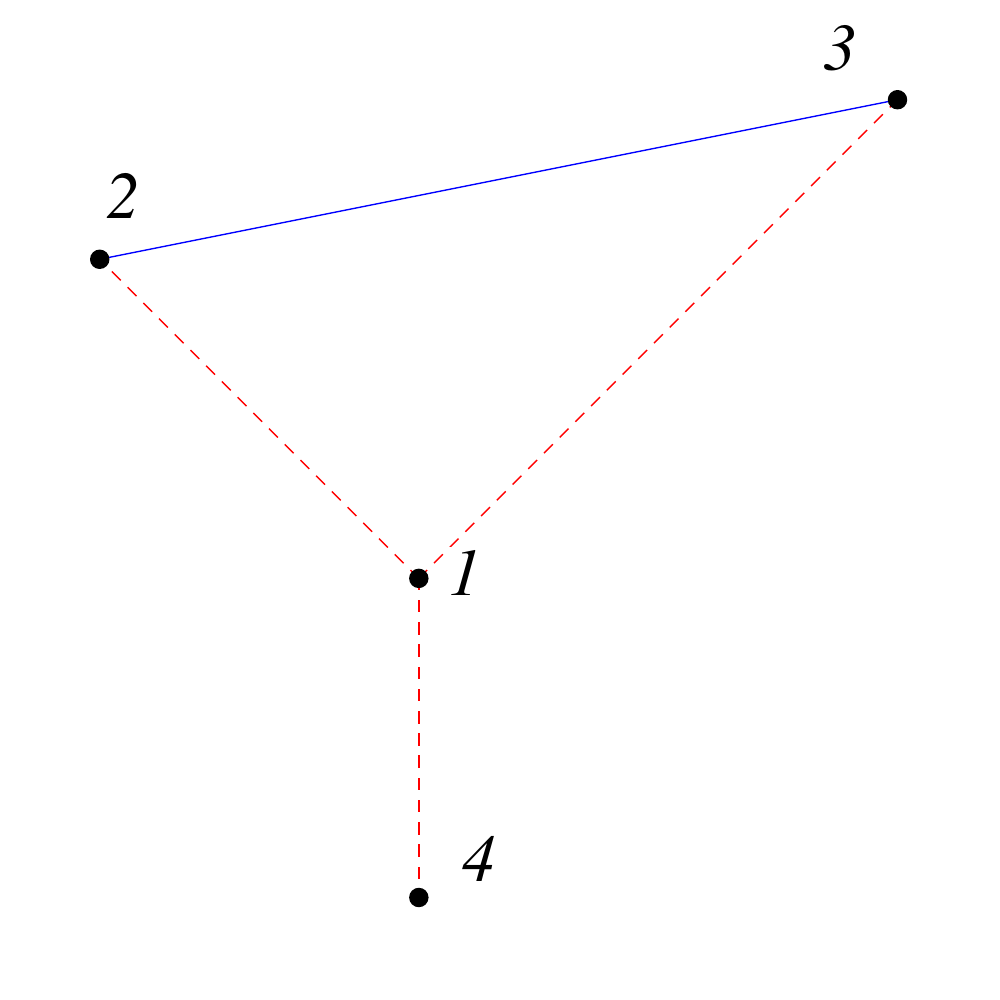}
\label{fig:graph1e}}
\caption{}
\end{figure}
For simplicity and without loss of generality, we may assume $s=1, i=2, j=3$ and $k=4$. We use Figure \ref{fig:graph1a}- \ref{fig:graph1e} to help explain the argument.
Assume $\p: V\ra \R^2$ is a generic realization, then $s(G)=s_{\p}(E)=|E|-rank(R_{G}(\p))$, and $s(G')=s_{\p}(E')=|E'|-rank(R_{G'}(\p))$. We start with Figure \ref{fig:graph1a}, which is supposed to be part of the generic realization $\p$. We change the realization by moving the image of $v_1$ to the line through $\p(v_2)$ and $\p(v_3)$(but different from $\p(v_2)$ and $\p(v_3)$), while keeping all the other vertices fixed.  Call the resulting realization $\p'$. By definition of generic realization, we have $rank(R_{G}(\p'))\leq rank(R_{G}(\p))$, hence $s_{\p'}(E)\geq s_{\p}(E)$. Figure \ref{fig:graph1b} illustrates part of $\p'$. 

Now we keep all the vertices fixed, but change the edge set to form a new graph    $G_1=(V, E_1, \p')$,  where $E_1=(E\cup\{e_{23}\})\ba\{e_{13}\}$. We claim that under the realization $\p'$,  there is a linear relation among $\vv_{12}, \vv_{23}$ and $\vv_{13}$. To see this, again by simplicity and without of generality, we assume the line through $\p'(v_{2})$ and $\p'(v_3)$ is not vertical.  Assume $\p'(v_1)=(x'_1,y'_1), \p'(v_2)=(x'_2, y'_2),\p'(v_3)=(x'_3, y'_3)$, then we can verify directly that
\[\frac{\vv_{12}}{x'_1-x'_2}+\frac{\vv_{23}}{x'_2-x'_3}-\frac{\vv_{13}}{x'_1-x'_3}=0.\]
So rows of $R_{G}(\p')$ and $R_{G_1}(\p')$ span the same space, hence have the same dimension. 
In Figure \ref{fig:graph1c}, we use red dashed line to denote the deleted edge and use blue line to denote the newly-added edge.

We then delete $e_{12}$ and $e_{14}$ to get Figure \ref{fig:graph1d}, this new graph is exactly $G'$. Although $\p'$ is not a generic realization, the similar argument in the proof of Deleting Lemma can be carried over to show that $s_{\p'}(G')=s_{\p'}(G_1)$. 

In the end, we move the image of $v_1$ back to $\p(v_1)$ to obtain the graph $G'$, which is illustrated in Figure \ref{fig:graph1e}. This obviously does not change the rank of rigidity matrix or the number of stress. Putting these together, we have
$$ s(G)=s_{\p}(G) \leq s_{\p'}(G) =s_{\p'}(G_1)
=s_{\p'}(G')=s_{\p}(G')=s(G').$$
This completes the proof. 
\end{proof}
\begin{remark}
The opposite operation of that in Proposition~\ref{prop: one extension}, i.e., get $G$ from $G'$, is called a one-extension.
\end{remark}

\begin{corollary}\label{cor:one extension}
Given $G=(V, E)$, assume there is a vertex $v_{s}$ with $\lambda(v_s)=4$ and the four edges containing $v_{s}$ are $e_{si}, e_{sj}, e_{sk}, e_{sl}$. Assume $e_{ij}\notin E$. We define a new graph $G'=(V, E')$ by taking $E'=(E\cup \{e_{ij}\})\ba\{e_{si}, e_{sj}, e_{sk}, e_{sl}\}$.  Then $s(G)\leq s(G')+1$.
\end{corollary}
\begin{proof}
We first delete the edge $e_{sl}$ to form a new graph $G_1$, then apply Proposition~\ref{prop: one extension} to $G_1$ and $v_s$. 
\end{proof}

We end this section with an example showing how Proposition~\ref{prop: one extension} might be used in determining the number of stress. This example will be used in the following sections.
\begin{figure}[!ht]
\centering
\subfigure[]{
\includegraphics[scale=0.4]{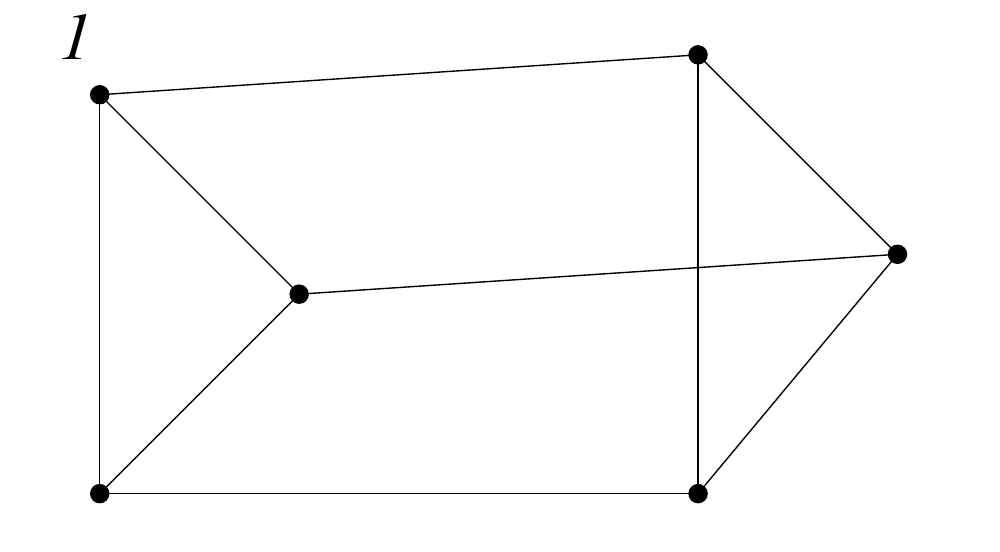}
\label{fig:graph2a}}
\subfigure[]{
\includegraphics[scale=0.4]{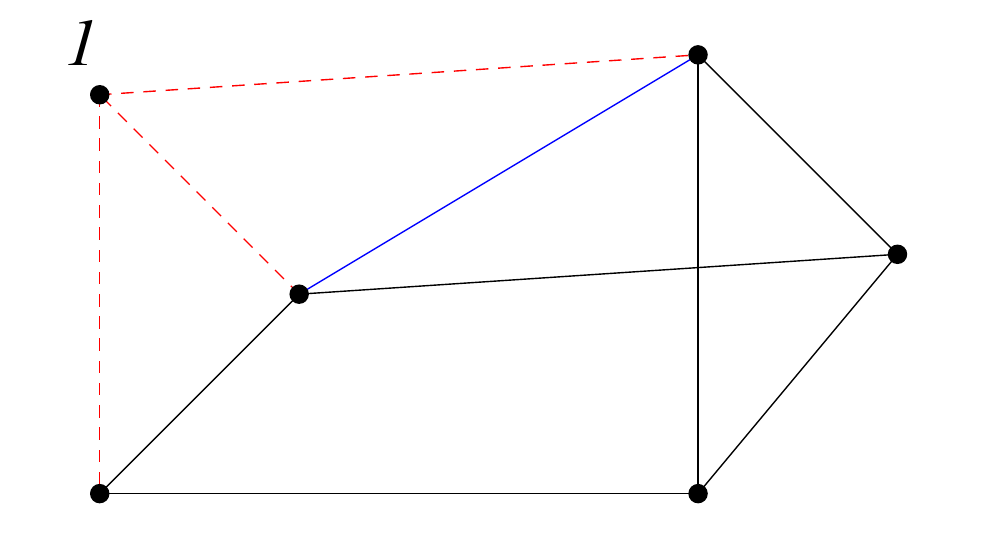}
\label{fig:graph2b}}
\subfigure[]{
\includegraphics[scale=0.4]{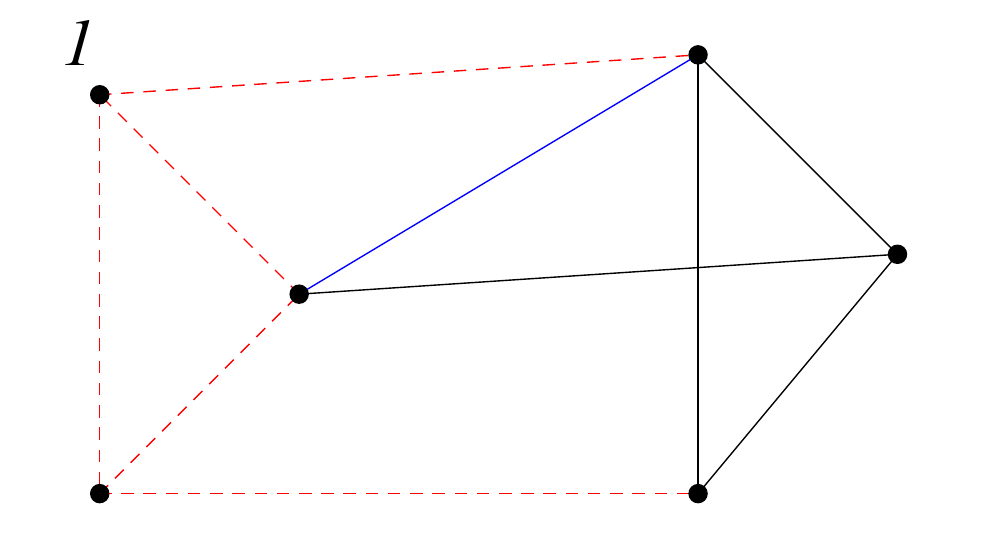}
\label{fig:graph2c}}
\caption{}
\end{figure}
\begin{example}\label{ex:three and three}
In Figure~\ref{fig:graph2a}, we showed one realization of a graph. We first apply Proposition~\ref{prop: one extension} to vertex $v_{1}$ to obtain Figure~\ref{fig:graph2b}. We use dashed red lines to denote the deleted edges and use blue line to denote the newly added line. We then repeated use the Deleting Lemma as we did in Example~\ref{ex:demon}. We are able to conclude that the original graph has number of stress $0$. 
\end{example}

\section{\bf Proof of Theorem 1: regular graph of degree four}\label{section:four valent graph}
Theorem 1 will be an easy corollary of the following lemma.
\begin{lemma}\label{lemma:fourValent}
Let $G=(V,E)$ be a graph and  $\p: V\rightarrow \R^2$ a generic planar realization. Assume each vertex of $G$ is of degree less than or equal to 4, and each connected component of $G$ contains at least one vertex of degree strictly less than 4. Then 
\begin{equation}\label{eq:four}
s(G)\leq \dfrac{n_{3}(G)+2n_4(G)+2c(G)}{5},
\end{equation}
where we use $n_{i}(G)$ to denote the number of vertices of $G$ of degree $i$, and use $c(G)$ to denote the number of connected components of $G$ that has at least one edge. 
\end{lemma}
\begin{proof}
Let $z(G)=\dfrac{n_{3}(G)+2n_4(G)+2c(G)}{5}$, we are going to use induction on the number of edges to show $s(G)\leq z(G)$. When $|E|=1$, $s(G)=0$ and \eqref{eq:four} obviously holds.  Now assume $|E|=e$ and  \eqref{eq:four} holds for any graph which satisfies the assumption of Lemma~\ref{lemma:fourValent} and whose edge set has size $< e$. 

If $G$ has a vertex whose degree is $0$, we can simply remove it since it does not affect either side of \eqref{eq:four}. Now we assume there is no such vertices. If $G$ is disconnected, then each connected component of $G$ still satisfies the assumption of Lemma ~\ref{lemma:fourValent} and has strictly less edges, hence \eqref{eq:four} holds for each of them by induction hypothesis.  Both sides of the \eqref{eq:four} are additive with respect to disjoint union of connected components, so it follows that \eqref{eq:four} would also hold for $G$. 

Now we assume $G$ is connected, we divide all situations into six cases. As we will have to talk about degree function for different graphs, we use $\lambda_{H}: V(H)\rightarrow \R$ to denote the degree function of a graph $H$. In particular, $\lambda=\lambda_{\Ga}$.

Case 1: $n_{1}(G)\neq 0$. 

Assume $\lambda(v_i)=1$ and $e_{ij}\in E$. Then we let $G' = (V', E')$ with $V'=V\ba\{v_i\}, E'=E\ba\{e_{ij}\}$.  Then it is obvious that $z(G')\leq z(G)$ and by Deleting Lemma we know $s(G') = s(G)$.  Since $|E'|<|E|$, it follows from the induction hypothesis that $s(G')\leq z(\Ga')$. So $s(\Ga)=s(\Ga')\leq z(\Ga')\leq z(\Ga)$. This completes the induction step. 
\vskip 2mm
Case 2: $n_{1}(G)=0$, but $n_{2}(G)\neq 0$. 

Assume $\la(v_{s})=2$ and $e_{sj}, e_{sk}\in E$. Define a new graph $\Ga_1=(V_1, E_1)$ by $V_{1}=V\ba\{v_s\}$ and $E_{1}=E\ba\{e_{sj}, e_{sk}\}$.  Define $H_{1}=(V^{H}_1, E^H_1)$ with $V^H_1=\{v_{s}, v_{j}, v_{k}\}$ and $E^H_1=\{e_{sj}, e_{sk}\}$. Define 
\[S_{1}=\{v\in V_{1}\cap V^H_1\big| \lambda_{\Ga_1}(v)\leq 2\}.\]
If $S_1=\emptyset$,  we let $H=H_1$ and $\Ga'=\Ga_1$. Otherwise, we define $\Ga_2=(V_2, E_2)$ by 
\[V_2= V_1\ba S_1 \text{\ \ and\ \ } E_2 = E_1\ba\{e\in E_1\big| e  \text{\ contains\ some\ vertex\ in\ } S_1\} \]
Define $H_{2}=(V^H_2, E^H_2)$ by 
\[V^H_{2}= V^H_1\cup \{v\in V_1\big| v \text{\ is\ connected\ to\ some\ point\ in\ } S_1 \text{\ by\ one\ edge\ in\ } E \}\]
and $E^H_2=E\ba E_2$.
Define 
\[S_{2}=\{v\in V_{2}\cap V^H_2\big| \lambda_{\Ga_2}(v)\leq 2\}.\]
If $S_2=\emptyset$, we let $H=H_2$ and $\Ga'=\Ga_2$. Otherwise we repeat the above steps to define $\Ga_3, H_3$ and $S_3$. This process would finally stop at some point and there is $\l\in \N$ such that
\[S_{l}=\emptyset, H=H_l \text{\ and\ } G'=\Ga_l.\]
\begin{figure}[ht]
\centering
\subfigure[$\Ga=(V,E)$]{
\includegraphics[scale=0.4]{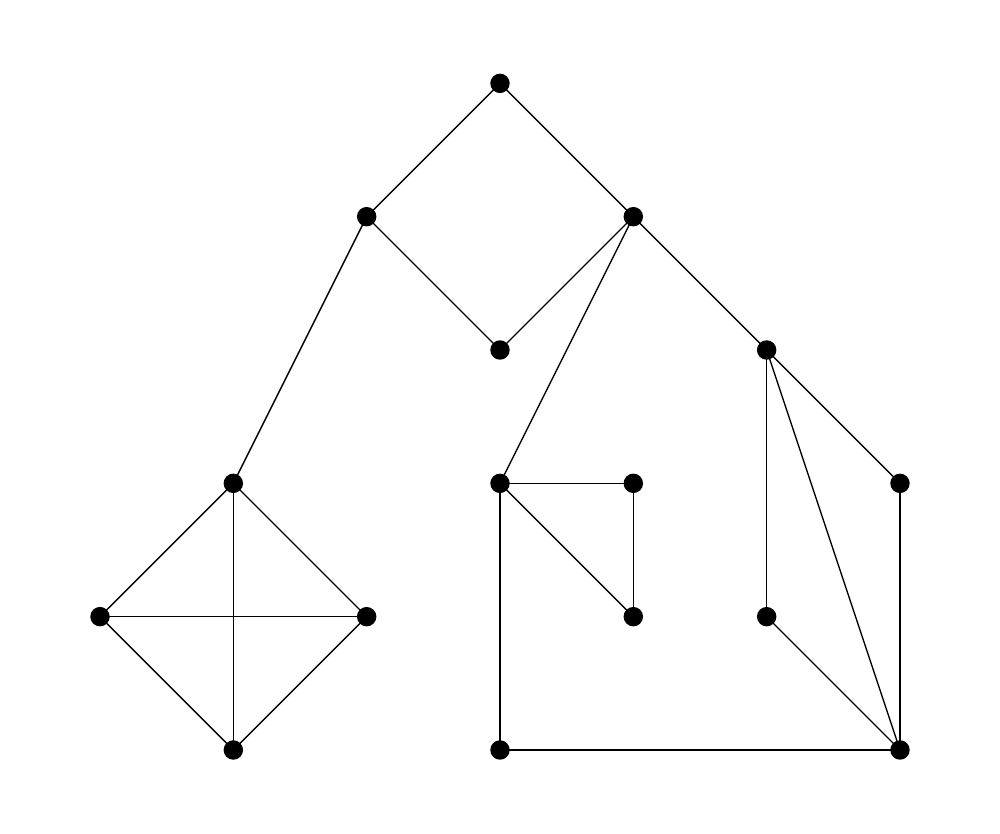}
\label{fig:graph3a}}
\subfigure[$E_1$ in dashed red, $E^H_1$ in blue, $S_1$ in orange]{
\includegraphics[scale=0.4]{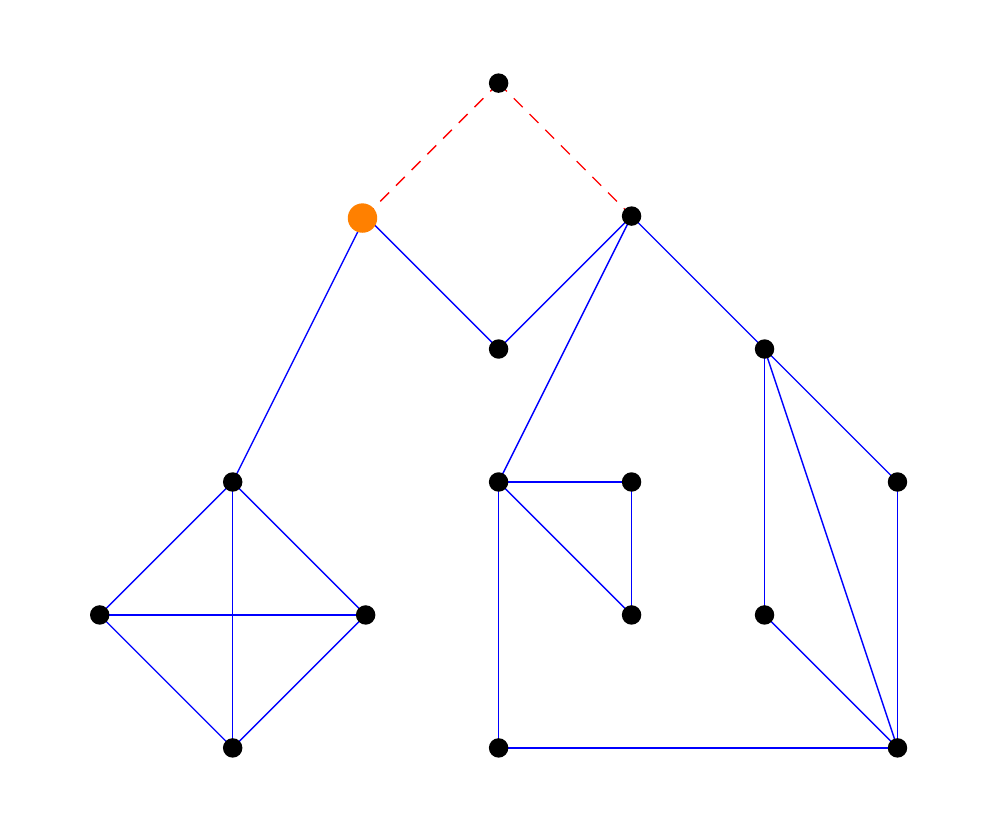}
\label{fig:graph3b}}
\subfigure[$E_2$ in dashed red, $E^H_2$ in blue, $S_2$ in orange]{
\includegraphics[scale=0.4]{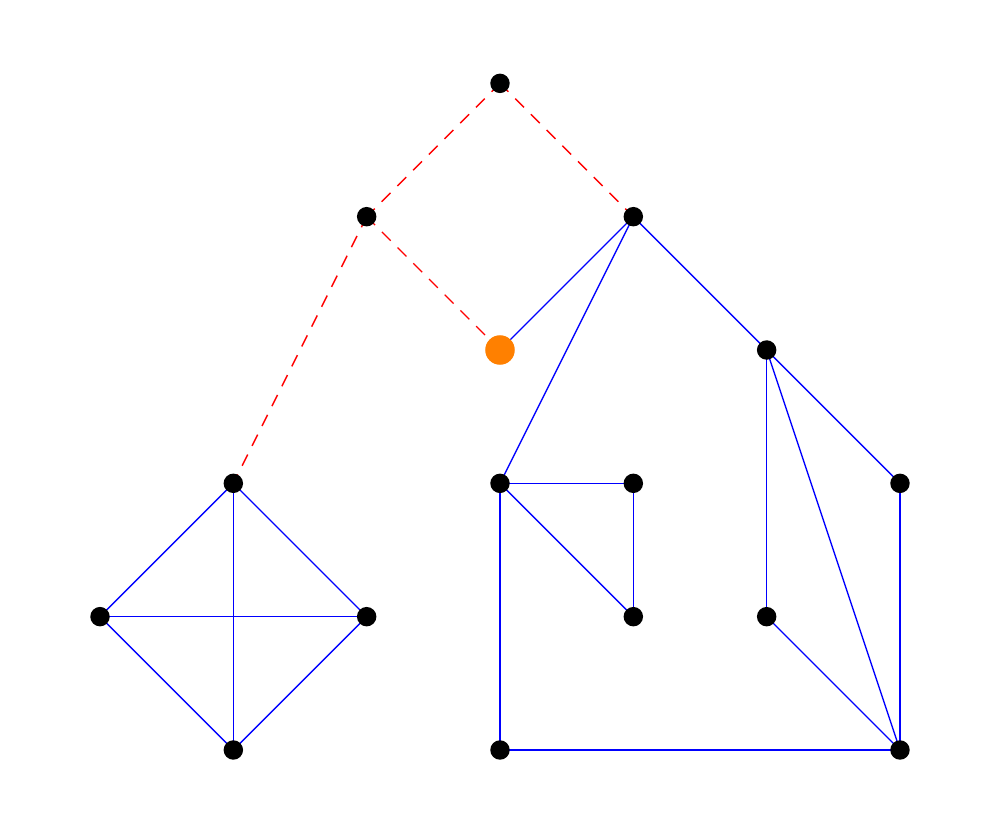}
\label{fig:graph3c}}
\subfigure[$E_3$ in dashed red, $E^H_3$ in blue, $S_3$ in orange]{
\includegraphics[scale=0.4]{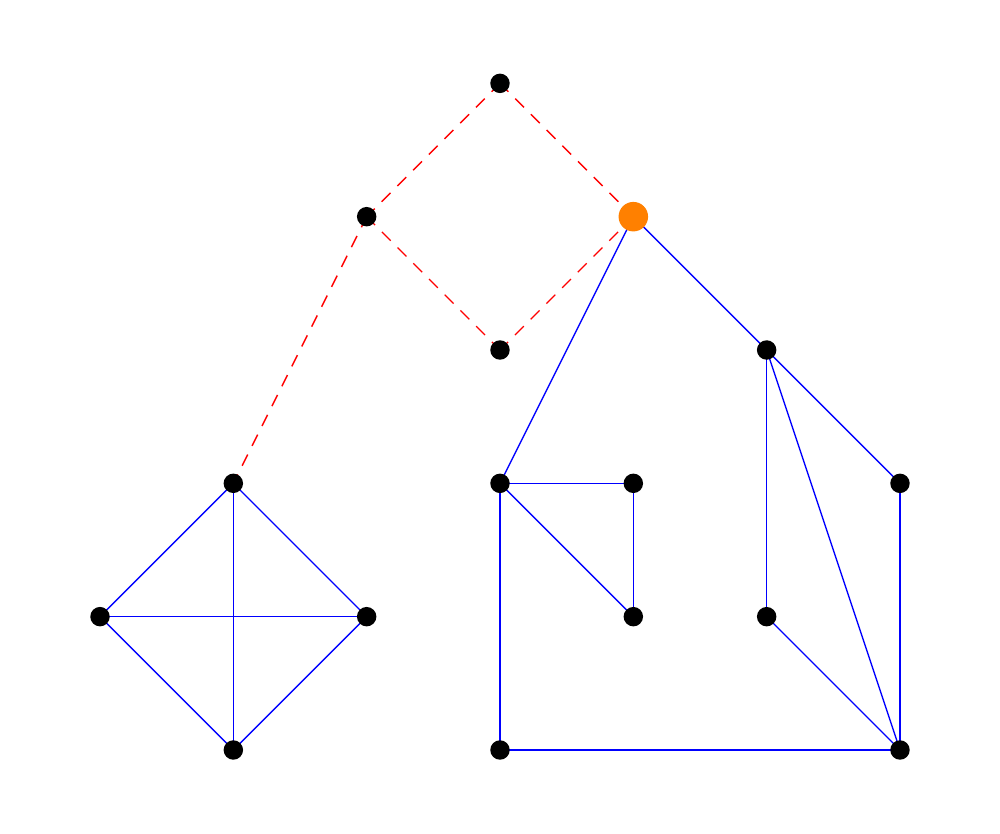}
\label{fig:graph3e}}
\subfigure[$E_4$ in dashed red, $H_4$ in blue, $S_4=\emptyset$]{
\includegraphics[scale=0.4]{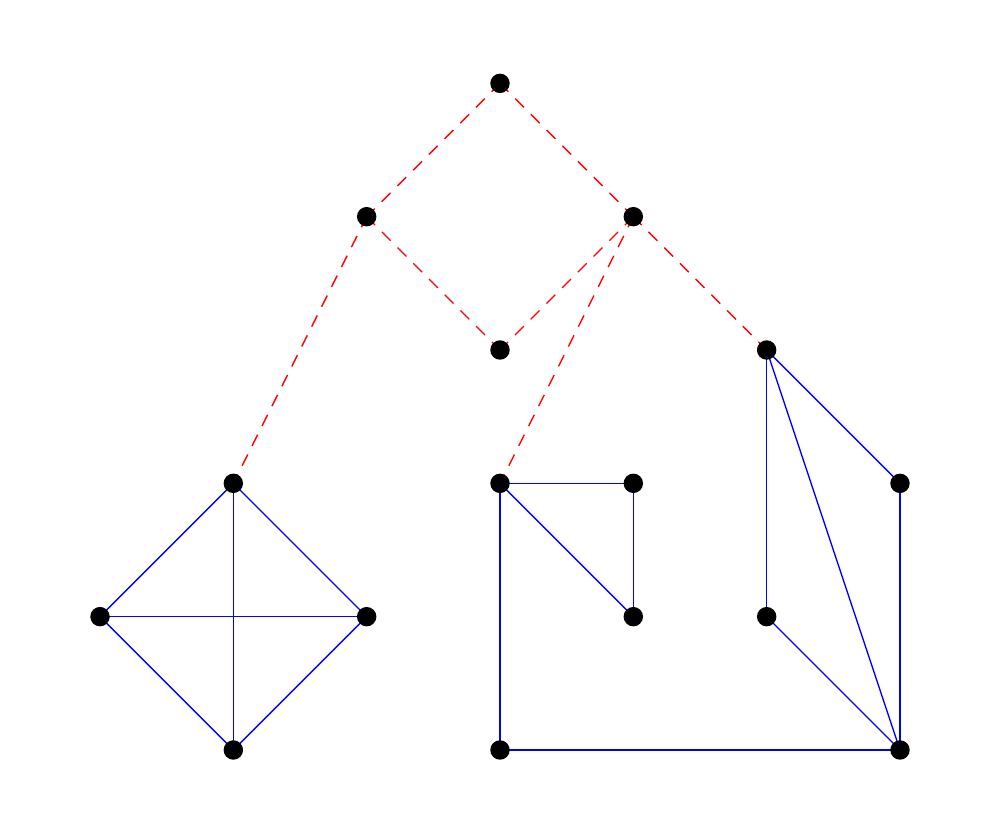}
\label{fig:graph3d}}
\caption{An example when there is a vertex of degree 2}
\end{figure}
Maybe the above process is best explained using an example, as we illustrated in Figure ~\ref{fig:graph3a} to ~\ref{fig:graph3e}. In that case, $l=4$.

 We use $V_H$ and $E_H$ to denote the vertex set and edge set of $H$ respectively, use $V'$ and $E'$ to denote the vertex set and edge set of $\Ga'$ respectively. Let $T=V_H\cap V'$. Suppose $|T|=t$. We make several observations:
 \begin{itemize}
 \item Each $H_{i}$, hence $H$, is connected. This can be shown inductively on $i$.
 \item $\Ga'$ maybe disconnected. The number of connected components, hence $c(\Ga')$, is less than or equal to $t$.
 \item For any $v\in T$, we have $\lambda_{G'}(v)=3$, $\lambda_{H}(v)=1$ and $\lambda(v)=4$.
 \item $s(\Ga')=s(\Ga)$.
 \end{itemize}
 Now we apply Lemma~\ref{lem:connect} to $H$ to get
 \[2 \bigg|\{v\in V_H\big| \lambda_{H}(v)=4\}\bigg|+\bigg|\{v\in V_H\big| \lambda_H(v)=3\}\bigg|-\bigg|\{v\in V_H\big| \lambda_H(v)=1\}\bigg|\geq -2,\]
 where we have used $|A|$ to denote the cardinality of $A$.
 It follows that 
 \[2\bigg|\{v\in V_H\ba T\big| \lambda(v)=4\}\bigg|+\bigg|\{v\in V_H\ba T\big| \lambda(v)=3\}\bigg|-t\geq -2.\]
So \[
\begin{array}{ll}
5(z(\Ga)-z(\Ga'))&=(2n_{4}(\Ga)+n_3(\Ga))-(2n_4(\Ga')+n_3(\Ga'))+2(1-c(\Ga'))\\
&= 2\bigg|\{v\in V_H\ba T\big| \lambda(v)=4\}\bigg|+\bigg|\{v\in V_H\ba T\big| \lambda(v)=3\}\bigg|+t+2(1-c(\Ga'))\\
&\geq -2+2t+2(1-c(\Ga'))\\
&\geq 0
\end{array}
\]
Since $\Ga'$ obviously has less edges than $\Ga$, by the induction hypothesis, $s(\Ga')\leq z(\Ga')$. Hence
\[s(\Ga)=s(\Ga')\leq z(\Ga')\leq z(\Ga).\]
This completes the induction step.
\vskip 2mm
Case 3: $n_{1}(G)=n_{2}(G)=0$ and one can disconnect the graph by deleting some one edge $e_{ij}$.

Define a new graph $\Ga_1=(V_1, E_1)$ with $V_1=V$ and $E_1=E\ba\{e_{ij}\}$. By the corollary of Disconnecting Lemma, Corollary~\ref{cor:disconnect1}, we have $s(\Ga)=s(\Ga_1)$. Now we find ourselves in the similar situation as in Case 2, repeat the argument there will lead us to a subgraph $\Ga'$ with less edges and same number of stress. Then $s(\Ga)=s(\Ga')\leq z(\Ga')\leq z(\Ga)$.  This completes the induction step. 
\vskip 2mm
Case 4: $n_{1}(G)=n_{2}(G)=0$, the graph would remain connected upon deleting any one edge, but will become disconnected upon deleting some two edges $e_{ij}$ and $e_{st}$. 

Define $\Ga'=(V',E')$ with $V'=V$ and $E'=E\ba\{e_{ij}\}$. According to Corollary~\ref{cor:disconnect2}, deleting both $e_{ij}$ and $e_{st}$ will not affect the number of stress, so deleting one of them certainly will not neither, hence $s(\Ga')=s(\Ga)$. And in this case we obviously have $z(\Ga')\leq z(\Ga)$ and $|E'|< |E|$. So 
\[s(\Ga)=s(\Ga')\leq z(\Ga')\leq z(\Ga).\]
This completes the induction step.
\vskip 2mm
Case 5: $n_{1}(G)=n_{2}(G)=0$, the graph would remain connected upon deleting any two edges. And there is a vertex $v_{s}$ of degree $3$ with $e_{si}, e_{sj}, e_{sk}\in E$ and $e_{ij}\notin E$. 

Define $\Ga'=(V',E')$ with $V'=V\ba\{v_s\}$ and $E'=(E\cup\{e_{ij}\})\ba\{e_{si}, e_{sj}, e_{sk}\}.$ We observe that $\Ga'$ is still connected, otherwise deleting $e_{sk}$ would disconnect $\Ga$, which contradicts the assumption.  So $z(\Ga')\leq z(\Ga)$. According to Proposition~\ref{prop: one extension}, $s(\Ga')\geq s(\Ga)$. Since $|E'|<|E|$, so the induction hypothesis applies. Putting all these together, we have 
\[s(\Ga)\leq s(\Ga')\leq z(\Ga')\leq z(\Ga).\]
This completes the induction step. 
\vskip 2mm
Case 6: All other cases. In this case, we must have $n_{1}(G)=n_{2}(G)=0$, the graph would remain connected upon deleting any two edges and there must be vertices $v_{s}, v_i, v_j, v_k$ with $\lambda(v_{s})=3$,  such that the complete graph on these four vertices is a subgraph of $G$.
 
 First we claim that either  $\lambda(v_i)=\lambda(v_j)=\lambda(v_{k})=3$, in which case $G$ is just a complete graph on four vertices and \eqref{eq:four} can be verified directly, or $\lambda(v_i)=\lambda(v_j)=\lambda(v_{k})=4$. Reason:
 \begin{itemize}
 \item If $\lambda(v_i)=4, \la(v_j)=\la(v_k)=3$, then the graph will become disconnected upon deleting the one edge containing $v_i$ other than $v_{is}, v_{ij}, v_{ik}$. 
 \item Similarly, if $\lambda(v_i)=\la(v_j)=4, \la(v_k)=3$, then we can delete two edges to disconnect the graph.
 \end{itemize}
 Now assume $\lambda(v_i)=\lambda(v_j)=\lambda(v_{k})=4$, then there should be $e_{ix}, e_{jy}, e_{kz}\in E$ that is not in $\{e_{ij}, e_{jk}, e_{ik}, e_{si}, e_{sj}, e_{sk}\}$. We are facing several sub cases here. We do not list all cases, but any other case would be equivalent to one of them.
 \begin{figure}[ht]
 \subfigure[]{
 \includegraphics[scale=0.4]{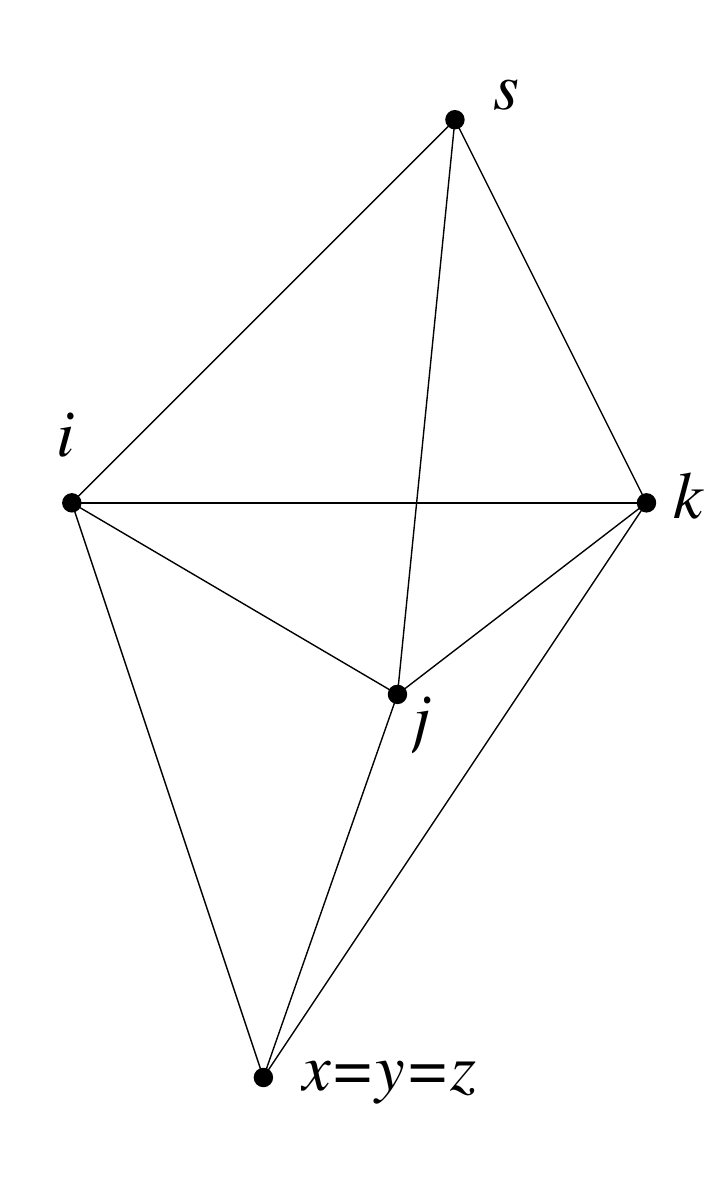}
 \label{fig:graph4a}}
 \subfigure[]{
 \includegraphics[scale=0.4]{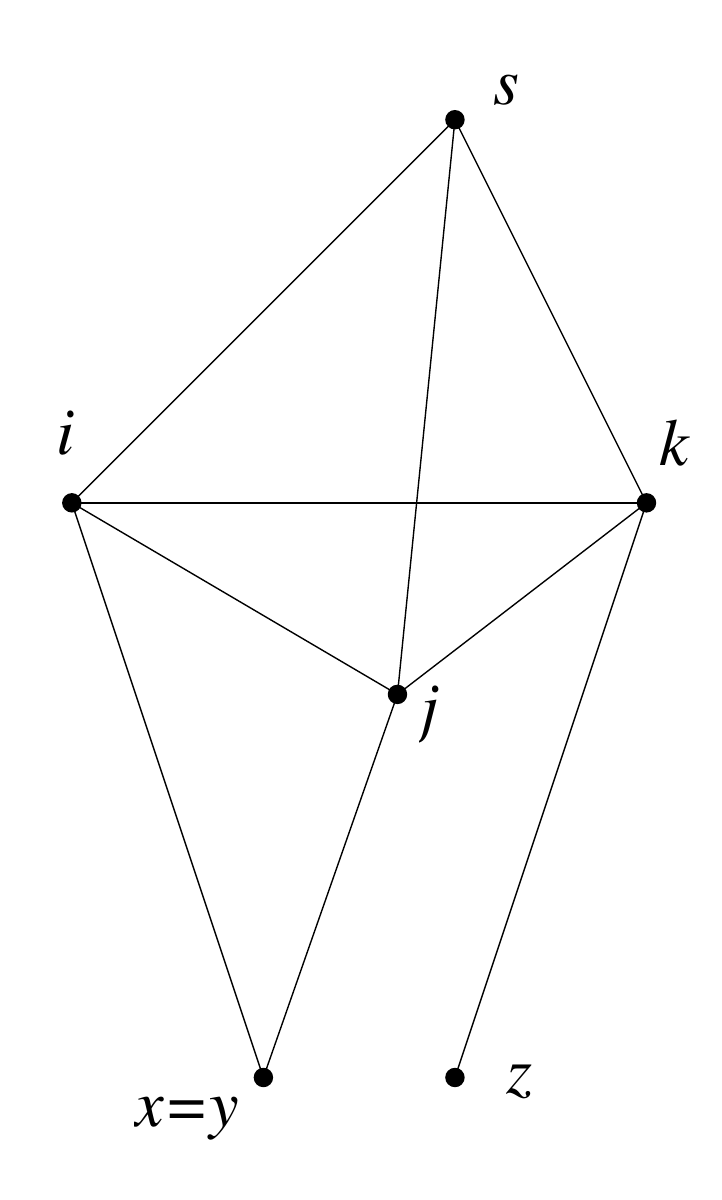}
 \label{fig:graph4b}}
 \subfigure[]{
 \includegraphics[scale=0.4]{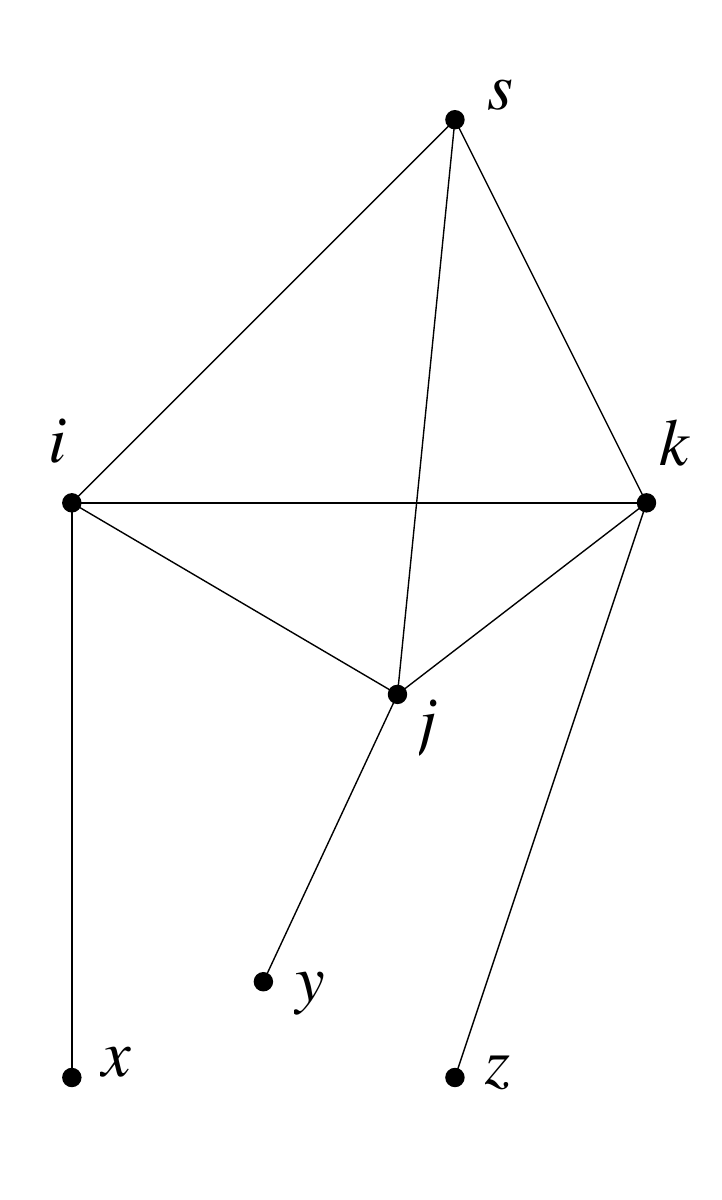}
 \label{fig:graph4c}}
 \caption{Several sub cases of Case $6$}
 \end{figure}
 Let $G'=(V', E')$ with $V'=V\ba \{v_s, v_i, v_j, v_k\}$, and  $E'=E\ba \{e_{si}, e_{sj}, e_{sk}, e_{ij}, e_{ik}, e_{jk}, e_{ix}, e_{jy}, e_{kz}\}$. 
 \begin{figure}[h!]
 \subfigure[]{
 \includegraphics[scale=0.45]{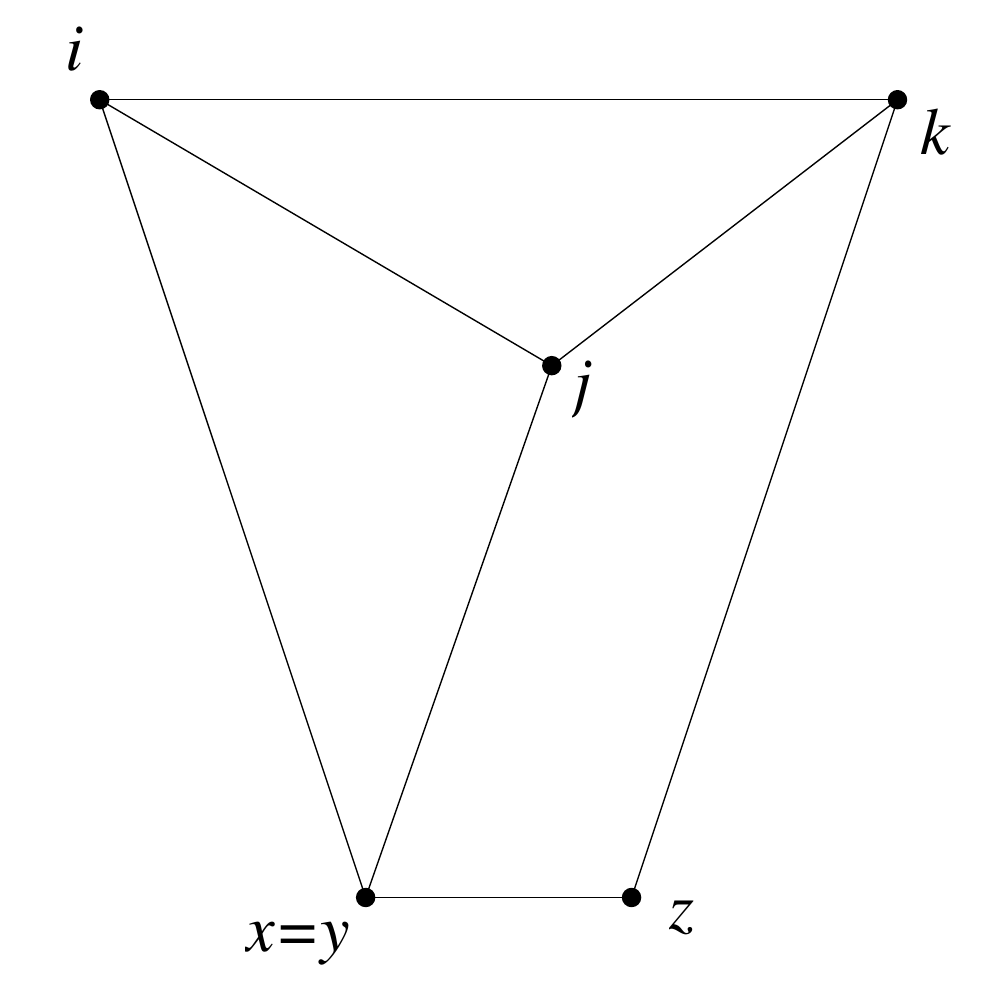}
 \label{fig:graph5a}}
 \subfigure[]{
 \includegraphics[scale=0.45]{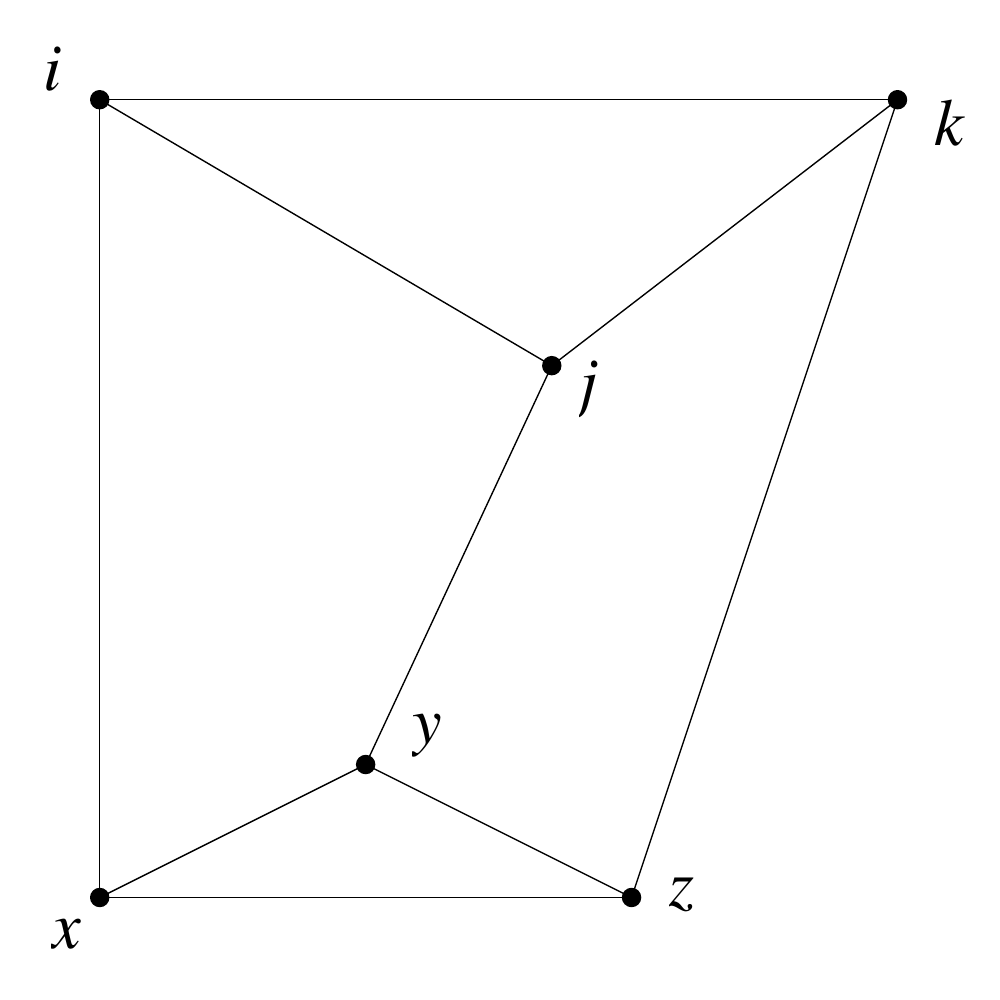}
 \label{fig:graph5b}}
 \caption{}
 \end{figure}
 \begin{enumerate}
 \item[(a)] As shown in Figure~\ref{fig:graph4a}, $x=y=z$. If $\la(v_x)=4$, then $G$ can be disconnected by deleting one edge, this contradicts the assumption of Case 6. So $\la(v_x)=3$ and $G$ is the graph as shown in Figure~\ref{fig:graph4a}. We can show directly that $s(G)=2=z(G)$.
 \item[(b)] As shown in Figure~\ref{fig:graph4b}, $x=y\neq z$. In this case, $\Ga'$ must be a connected graph, otherwise $\Ga$ can be disconnected by deleting one or two edges. Since the graph shown in Figure~\ref{fig:graph5a} has number of stress $0$, which can be shown by repeatedly using the Deleting Lemma, we can apply the Disconnecting Lemma~\ref{lem:disc} to $G$ to obtain $s(G)=s(\Ga')+1$. Here we have used the fact the the complete graph on four vertices has number of stress $1$. We can see by direct counting that $z(\Ga)-z(\Ga')>1$. And since $|E'|<|E|$, the induction hypothesis applies to $\Ga'$. Putting all these together, we have
 \[s(\Ga)=s(\Ga')+1\leq z(\Ga')+1<z(\Ga).\]
 This completes the induction step.
 \item[(c)] As shown in Figure~\ref{fig:graph4c}, $x, y, z$ are three distinct points. This case is very similar to the previous one. $\Ga'$ has to be connected, otherwise $\Ga$ can be disconnected by deleting one or two edges. The graph in Figure~\ref{fig:graph5b} has number of stress $0$, as shown in Example~\ref{ex:three and three}, so we may apply Disconnecting Lemma~\ref{lem:disc} to $G$ to obtain $s(G)=s(\Ga')+1$.  The rest is the same as in previous case.
 \end{enumerate}
 \vskip 2mm
The proof is now complete.
\end{proof}
Now Theorem 1 follows easily.
\begin{proof}[proof of Theorem 1]
Delete any one edge from $G$ to obtain a new graph $G'=(V,E')$, then $s(G)\leq s(G')+1$.  It is easy to see that $n_{3}(G')=2, n_{4}(G')=m-2$ and $c(G')=1$ since each connected component of $G'$ has to contain an even number of vertices of odd degree.  Now we apply Lemma~\ref{lemma:fourValent} to $G'$, we get 
\[s(G')\leq \dfrac{n_{3}(G')+2n_{4}(G')+2c(G')}{5}=\dfrac{2+2(m-2)+2}{5}=\dfrac{2m}{5},\]
hence
\[s(G)\leq \dfrac{2m}{5}+1.\]
So $r(G)\geq 2m-s(G)\geq \dfrac{8m}{5}-1$.
\end{proof}
\section{\bf Proof of Theorem 2: regular graph of degree five}\label{section:five valent graph}
The proof of Theorem 2 uses the same ideas as that of Theorem 1.  We first prove a lemma that is parallel to Lemma~\ref{lemma:fourValent}.
\begin{lemma}\label{lemma:fiveValent}
Let $G=(V,E)$ be a graph and $\p: V\rightarrow \R^2$ a generic planar realization. Assume each vertex of $G$ is of valency less than or equal to 5, and each connected component of $G$ contains at least one vertex of valency strictly less than 5. Then 
\begin{equation}\label{eq:four}
s(G)\leq \dfrac{n_{3}(G)+2n_4(G)+3n_5(G)+2c(G)}{18/5},
\end{equation}
where we have used the same notation as that in Lemma~\ref{lemma:fourValent}.
\end{lemma}
\begin{proof}
Let $z(G)=\dfrac{n_{3}(G)+2n_4(G)+3n_5(G)+2c(G)}{18/5}$, we are going to use induction on the number of edges to show $s(G)\leq z(G)$. When $|E|=1$, $s(G)=0$ and \eqref{eq:four} obviously holds.  Now assume $|E|=e$ and  \eqref{eq:four} holds for any graph which satisfies the assumption of Lemma~\ref{lemma:fiveValent} and whose edge set has size $< e$.  Now we may assume $G$ is connected as we did in the proof of Lemma~\ref{lemma:fourValent}. We divided all situations into six cases.

Case 1: $n_{1}(G)\neq 0$.

Same as Case 1 in the proof of Lemma~\ref{lemma:fourValent}.
\vskip 2mm
Case 2: $n_{1}(G)= 0$, but $n_{2}(G)\neq 0$.

In this case, we need small modification from Case 2 in the proof of Lemma~\ref{lemma:fourValent}.  Define $\Ga', H, T$ as we did there.  Suppose $|T|=t$, the several observation we made there still hold except the third one, which we will change to:
\begin{itemize}
\item $T=X\cup Y$, where $X=\{v\in T\big| \lambda_{H}(v)=2, \lambda_{\Ga'}(v)=3\}$ and $Y=\{v\in T\big| \lambda_{H}(v)=1, \lambda_{\Ga'}(v)=4\text{ or }3\}$.
\end{itemize}
Let $x=|X|, y=|Y|$, then $x+y=t$.  To simplify notation, for any subset $U\subset V$, we let
\[U^{i}=\{v\in U\big| \lambda(v)=i\}.\]
 Apply Lemma~\ref{lem:connect} to $H$ to get
\[3n_{5}(H)+ 2n_{4}(H)+n_{3}(H)-n_{1}(H)\geq -2.\]
It follows that
 \[3\big|(V_H\ba T)^5\big|+2\big|(V_H\ba T)^4\big|+\big|(V_H\ba T)^3\big|-y\geq -2.\]
So \[
\begin{array}{ll}
\dfrac{18}{5}(z(\Ga)-z(\Ga'))&=(3n_{5}(\Ga)+2n_{4}(\Ga)+n_3(\Ga))-(3n_{5}(\Ga')+2n_4(\Ga')+n_3(\Ga'))+2(1-c(\Ga'))\\
&= 3\big|(V_H\ba T)^5\big|+2\big|(V_H\ba T)^4\big|+\big|(V_H\ba T)^3\big|+2x+y+2(1-c(\Ga'))\\
&\geq -2+y+2x+y+2(1-c(\Ga'))\\
&=2t-2c(\Ga')\\
&\geq 0
\end{array}
\]
Since $\Ga'$ obviously has less edges than $\Ga$, by the induction hypothesis, $s(\Ga')\leq z(\Ga')$. Hence
\[s(\Ga)=s(\Ga')\leq z(\Ga')\leq z(\Ga).\]
This completes the induction step.
\vskip 2mm
Case 3:  $n_{1}(G)=n_{2}(G)= 0$, and one can disconnect the graph by deleting up to two edges.

This is the same as Case 3 and Case 4 in the proof of Lemma~\ref{lemma:fourValent}.
\vskip 2mm
 In all of the following cases, we assume $n_{1}(G)=n_{2}(G)= 0$ and $G$ will remain connected upon deleting any two edges.
 \vskip 2mm

Case 4: $n_3(\Ga)\neq 0$.

Assume $v_s\in V$ such that $\lambda(v_s)=3$ and $e_{si}, e_{sj}, e_{sk}\in E$. Now define $\Ga'=(V', E')$ by $V'=V\ba\{v_s\}$ and $E'=E\ba\{e_{si}, e_{sj}, e_{sk}\}$. 
By  the Corollary~\ref{cor:delete} of Deleting Lemma, we have $s(\Ga')\geq s(\Ga)-1$.
We notice that $\Ga'$ is still connected, otherwise $\Ga$ can be disconnected by deleting one edge, which violates our assumption. So by simple counting we see that $z(\Ga)-z(\Ga')=\dfrac{4}{18/5}>1$. Since $|E'|<|E|$, the induction hypothesis applies. Putting these together, we have 
\[s(\Ga)\leq s(\Ga')+1\leq z(\Ga')+1<z(\Ga).\]
This completes the induction step.
\vskip 2mm
Case 5: $n_{3}(\Ga)=0$, and exists vertex $v_s$ with $\lambda(v_s)=4$ and $e_{si}, e_{sj}, e_{sk}, e_{sl}\in E, e_{ij}\notin E$. 

Define $\Ga'=(V', E')$ by $V'=V\ba\{v_s\}$ and $E'=(E\cup \{e_{ij}\})\ba\{e_{si}, e_{sj}, e_{sk}, e_{sl}\}$. Then according to Corollary~\ref{cor:one extension}, $s(\Ga')\geq s(\Ga)-1$. Notice that $\Ga'$ is still connected, otherwise $\Ga$ can be disconnected by deleting one or two edges, which contradicts our assumption. By direct counting we can see $z(\Ga)-z(\Ga')=\dfrac{4}{18/5}>1$. Since $|E|'<|E|$, the induction hypothesis applies. Putting these together, we have
\[s(\Ga)\leq s(\Ga')+1\leq z(\Ga')+1<z(\Ga).\]
This completes the induction step. 
\vskip 2mm
Case 6: $n_{3}(\Ga)=0$, and exists vertex $v_{s}$ with $\la(v_s)=4$, and vertices $v_{i}, v_j, v_k, v_l$ such that the complete graph on these five vertices is a subgraph of $G$. 

First note that either $\lambda(v_i)=\lambda(v_j)=\lambda(v_k)=\lambda(v_l)=4$, in which case $\Ga$ is a complete graph on five vertices and we can verify by direct computation that $s(\Ga)=3 <\dfrac{12}{18/5}=z(\Ga)$, or  at least three of the four vertices $\{v_i, v_j, v_k, v_l\}$ must have degree $5$, otherwise $\Ga$ can be disconnected by deleting one or two edges.  Assume $\lambda(v_i)=\lambda(v_j)=\lambda(v_k)=5$ and $e_{ix}, e_{jy}, e_{kz}\in E$.  If $\la(v_l)=4$, then we are in similar situation as that in Case 6(b),(c) in the proof of Lemma~\ref{lemma:fourValent}(cannot be in Case 6(a)), we can define $\Ga'=(V', E')$ with $V'=V$ and $E'=E\ba\{e_{ix}, e_{jy}, e_{kz}\}$ to complete the induction step.

Now assume $\lambda(v_{l})=5$ and $e_{lw}\in E$. There are two sub cases:
\begin{itemize}
\item[(a)] $x=y=z=w$, then we must have $\lambda(x)=4$, otherwise $\Ga$ can be disconnected by deleting the other edge containing $x$. In this case, one can show by direct computation that $s(\Ga)=5=z(\Ga).$
\item[(b)]$x, y, z, w$ are not all equal. Then define $\Ga'=(V', E')$ with $V'=V$,and $E'=E\ba\{e_{ix}, e_{jy}, e_{kz}, e_{lw}\}$. Then by a similar argument as that in Case 6(b),(c) in the proof of Lemma~\ref{lemma:fourValent}, we can show $s(\Ga')\geq s(\Ga)-1$ and $z(\Ga)-z(\Ga')>1$, so 
\[s(\Ga)\leq s(\Ga')+1\leq z(\Ga')+1<z(\Ga).\]
\end{itemize}
\vskip 2mm
The proof is now complete.
\end{proof}

Now Theorem 2 follows by a similar simple argument as that in the proof of Theorem 1. We omit it here.

\end{document}